\documentclass[12pt]{amsart}
\usepackage{amsmath,amssymb,amsthm}
\usepackage{latexsym}%,ulem}      % for the qed symbol
\usepackage{color,colordvi,graphicx,multirow}
\numberwithin{equation}{section}
\newtheorem{theorem}{Theorem}
\newtheorem{proposition}[theorem]{Proposition}
\newtheorem{lemma}[theorem]{Lemma}
\newtheorem{corollary}[theorem]{Corollary}

\theoremstyle{remark}

\newtheorem{remark}[theorem]{Remark}

\newcounter{FNC}[page]
\def\fauxfootnote#1{{\addtocounter{FNC}{2}$^\fnsymbol{FNC}$%
     \let\thefootnote\relax\footnotetext{$^\fnsymbol{FNC}$\Magenta{#1}}}}
\newcommand{\defcolor}[1]{\Maroon{#1}} 
\newcommand{\demph}[1]{\defcolor{{\sl #1}}}

\newcommand{\C}{{\mathbb C}}

\newcommand{\R}{{\mathbb R}}
\newcommand{\Q}{{\mathbb Q}}
\renewcommand{\P}{{\mathbb P}}

\newcommand{\url}[1]{{\tt #1}}

\newcommand{\Fdot}{F_{\bullet}}

\newcommand{\blambda}{\boldsymbol{\lambda}}

\newcommand{\Gr}{\mbox{\rm Gr}}
\newcommand{\Id}{\mbox{\rm Id}}
\newcommand{\Mat}{\mbox{\rm Mat}}
\newcommand{\rowspace}{\mbox{\rm row space}}

\newcommand{\rank}{\mbox{\rm rank}}
\newcommand{\Span}{\mbox{\rm span}}
\newcommand{\Vdm}{\mbox{\rm Vdm}}

\DeclareRobustCommand{\I}{\includegraphics{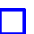}}
\DeclareRobustCommand{\Ig}{\includegraphics{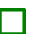}}
\DeclareRobustCommand{\Is}{\includegraphics{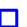}}
\DeclareRobustCommand{\IIs}{\includegraphics{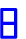}}
\DeclareRobustCommand{\IIIs}{\includegraphics{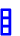}}
\DeclareRobustCommand{\III}{\includegraphics{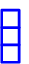}}
\DeclareRobustCommand{\TI}{\includegraphics{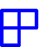}}
\DeclareRobustCommand{\TIs}{\includegraphics{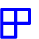}}
\DeclareRobustCommand{\TIgs}{\includegraphics{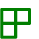}}
\DeclareRobustCommand{\TT}{\includegraphics{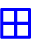}}
\DeclareRobustCommand{\TTs}{\includegraphics{pictures/22s.eps}}
\DeclareRobustCommand{\TTT}{\includegraphics{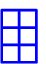}}
\DeclareRobustCommand{\Th}{\includegraphics{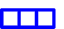}}
\DeclareRobustCommand{\Ths}{\includegraphics{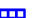}}
\DeclareRobustCommand{\ThI}{\includegraphics{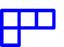}}
\DeclareRobustCommand{\ThIs}{\includegraphics{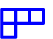}}
\DeclareRobustCommand{\ThIIs}{\includegraphics{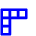}}
\DeclareRobustCommand{\ThIIb}{\includegraphics{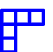}}
\DeclareRobustCommand{\ThII}{\includegraphics{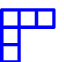}}
\DeclareRobustCommand{\ThTh}{\includegraphics{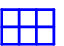}}
\DeclareRobustCommand{\ThThT}{\includegraphics{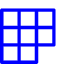}}
\DeclareRobustCommand{\ThThTs}{\includegraphics{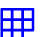}}
\DeclareRobustCommand{\ThThThs}{\includegraphics{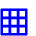}}
\DeclareRobustCommand{\ThThTh}{\includegraphics{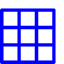}}
\DeclareRobustCommand{\F}{\includegraphics{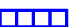}}
\DeclareRobustCommand{\Fs}{\includegraphics{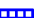}}
\DeclareRobustCommand{\FF}{\includegraphics{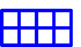}}
\DeclareRobustCommand{\V}{\includegraphics{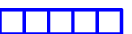}}
\DeclareRobustCommand{\minusIt}{\includegraphics{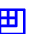}}
\DeclareRobustCommand{\minusIs}{\includegraphics{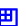}}

\newcommand{\Skew}[8]{\begin{picture}(41,31)(-3,-2.5)
 \put(-3,-2.5){\includegraphics{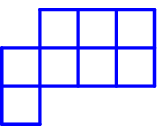}}
                        \put(11,22){\small$#1$}\put(22,22){\small$#2$}\put(33,22){\small$#3$}
 \put( 0,11){\small$#4$}\put(11,11){\small$#5$}\put(22,11){\small$#6$}\put(33,11){\small$#7$}
 \put( 0, 0){\small$#8$}
\end{picture}}

\newcommand{\TheFo}[6]{\begin{picture}(61,21)(-3,-2.5)
 \put(-3,-2.5){\includegraphics{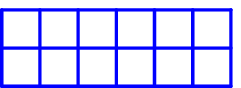}}
  \put( 0,11){\small$1$}\put(11,11){\small$1$}\put(22,11){\small$1$}
  \put(33,11){\small$#1$}\put(44,11){\small$#2$}\put(55,11){\small$#3$}
  \put( 0, 0){\small$#4$}\put(11, 0){\small$#5$}\put(22, 0){\small$#6$}
  \put(33, 0){\small$4$}\put(44, 0){\small$4$}\put(55, 0){\small$4$}
\end{picture}}

\title[Lower Bounds in Real Schubert Calculus]
 {Lower bounds in real Schubert Calculus}

\author{Nickolas Hein}
\address{Nickolas Hein \\
         Department of Mathematics\\
         University of Nebraska at Kearney\\
         Kearney\\
         Nebraska \ 68849\\
         USA}
\email{heinnj@unk.edu}
\urladdr{http://www.unk.edu/academics/math/faculty/About\_Nickolas\_Hein/}
\author{Christopher J. Hillar} 
\address{Christopher J. Hillar\\ Redwood Center for Theoretical Neuroscience \\ 
        University of California, Berkeley \\ 
         Berkeley, CA 94720}%\\         USA} 
\email{chillar@msri.org} 
\urladdr{http://www.msri.org/people/members/chillar/} 
\author{Frank Sottile}
\address{Frank Sottile \\
         Department of Mathematics\\
         Texas A\&M University\\
         College Station\\
         Texas \ 77843\\
         USA}
\email{sottile@math.tamu.edu}
\urladdr{\url{http://www.math.tamu.edu/~sottile}}
\thanks{Research of Sottile and Hein supported in part by NSF grants DMS-1001615 and  DMS-0922866.  Hillar supported by NSF grant IIS-1219212.}
\thanks{The work represents part of Hein's 2013 Ph.D.\ thesis from Texas A\&M University.}
\thanks{This material is based upon work supported by the National Science 
Foundation under Grant No. 0932078 000, while Sottile was in 
residence at the Mathematical Science Research Institute (MSRI) in 
Berkeley, California, during the winter semester of 2013.}
\thanks{Some computations done on computers purchased with NSF SCREMS grant DMS-0079536.}
\begin{document}
%%%%%%%%%%%%%%%%%%%%%%%%%%%%%%%%%%%%%%%%%%%%%%%%%%%%%%%%%%%%%%%%%%%%%%%%%%

%%%%%%%%%%%%%%%%%%%%%%%%%%%%%%%%%%%%%%%%%%%%%%%%%%%%%%%%%%%%%%%%%%%%%%%%%%
%
\begin{abstract}
 We describe a large-scale computational experiment to study 
 structure in the numbers of real solutions to osculating instances of Schubert
 problems. 
 This investigation uncovered Schubert problems whose computed numbers of real solutions
 variously exhibit nontrivial upper bounds, lower bounds, gaps, and a congruence modulo
 four. 
 We present a family of Schubert problems, one in each Grassmannian, and prove their
 osculating instances have the observed lower bounds and gaps.
\end{abstract}
%%%%%%%%%%%%%%%%%%%%%%%%%%%%%%%%%%%%%%%%%%%%%%%%%%%%%%%%%%%%%%%%%%%%%%%%%%

\maketitle

%%%%%%%%%%%%%%%%%%%%%%%%%%%%%%%%%%%%%%%%%%%%%%%%%%%%%%%%%%%%%%%%%%%%%%%%%%
%
\section*{Introduction}
%%%%%%%%%%%%%%%%%%%%%%%%%%%%%%%%%%%%%%%%%%%%%%%%%%%%%%%%%%%%%%%%%%%%%%%%%%
A remarkable recent story in mathematics was the proof of the Shapiro Conjecture 
(in real algebraic geometry) by Mukhin, Tarasov, and
Varchenko~\cite{MTV09} using methods from integrable systems.
Its simplest form involves the Wronski map, which sends a $k$-dimensional complex linear subspace
of univariate polynomials of degree $n{-}1$ to its Wronskian, a polynomial of degree
$k(n{-}k)$. 
In this context, the Mukhin-Tarasov-Varchenko Theorem states that if a polynomial
$w(t)$ of degree $k(n{-}k)$ has all of its roots real, then {\it every} $k$-plane of
polynomials with Wronskian $w(t)$ is real (i.e., has a basis of real polynomials).

The Wronskian is a map from a Grassmannian to a projective space, both of dimension
$k(n{-}k)$. 
Eremenko and Gabrielov~\cite{EG02} considered the real Wronski map that sends the real
Grassmannian to real projective space, computing its topological degree (actually the
degree of a lift to oriented double covers).
This topological degree is strictly positive when $n$ is odd,
so that for $n$ odd, there are always real $k$-planes of polynomials with given real
Wronskian, proving a weak version of the Shapiro Conjecture. 

The full Shapiro Conjecture went far beyond the reality of the Wronski map.
It concerned, more generally, intersections of Schubert varieties given by flags osculating a rational
normal curve (osculating instances of Schubert problems), positing that if the
osculating points were all real, then all of the points of intersection were also real.
When the Schubert varieties are all hypersurfaces, the conjecture
asserted that the fibers of the Wronski map over polynomials with all roots real contained
only real subspaces.
Initially considered too strong to be true, the Shapiro Conjecture came to be accepted due
to significant computer experimentation~\cite{Sot00,Ve00} and partial
results~\cite{EG02b,So99}.

Fibers of the Wronski map over a polynomial $w(t)$ with distinct roots are intersections of
hypersurface Schubert varieties given by flags osculating the rational normal curve
at the roots of $w(t)$.
When $w(t)$ is real, its roots form a real variety (stable under complex
conjugation) and the corresponding intersection of Schubert varieties is also real.
Eremenko and Gabrielov's topological degree is a lower bound for the number of real
points in that real intersection.
In related work, Azar and Gabrielov~\cite{AG} proved a lower bound for the number of
real rational functions  of degree $d$ with $2d{-}3$ real critical points and two real
points where the function values coincide.  
(This is a lower bound for a family of Schubert problems on a flag manifold given by
osculating flags and was motivated by data from the
experiment~\cite{RSSS}.) 
These results suggested the possibility of lower bounds for the number of real points 
in an intersection of Schubert varieties given by flags osculating the rational normal
curve, when the intersection is a real variety.

A preliminary investigation~\cite{Orig_Lower} confirmed this possibility and uncovered
other structures in the numbers of real solutions, including upper bounds, lower
bounds, gaps, and a congruence modulo four, in different families of Schubert problems.
Those data led to two papers~\cite{Pu13,HSZ13} which proved some of
the observed structure. 

We describe the design, execution, and some results of a large-scale computer
experiment~\cite{Lower_Exp} to study such real osculating instances of Schubert
problems. 
This study investigated over 344 million instances of 756 Schubert problems, and it used
over 549 gigahertz-years of computing.
The topological lower bounds of Eremenko and Gabrielov~\cite{EG02} apply to variants of
the Wronski map and were extended by Soprunova and Sottile~\cite{SS06} to give
topological lower bounds to osculating Schubert problems where at most two Schubert
varieties were not hypersurfaces.
We studied 273 such osculating Schubert problems and observed that these topological lower
bounds were sharp for all except six of them.

Four of these six continue to defy explanation.
For the remaining two, the lack of sharpness is due to a congruence
modulo four observed in both~\cite{Orig_Lower} and~\cite{Lower_Exp} for certain
symmetric Schubert problems.
This congruence has since been established by Hein and
Sottile in collaboration with Zelenko.
They first~\cite{HSZ13} treated osculating Schubert problems and
established a weak form of the congruence.
Later, they showed that many symmetric
Schubert problems in a Grassmannian given by isotropic flags (symplectic or
orthogonal) have a congruence modulo four on their number of real
solutions~\cite{HSZ_new}.

For example, Table~\ref{T:intro} summarizes the computation
for two Schubert problems in $\Gr(4,8)$, each with twelve solutions.
%%%%%%%%%%%%%%%%%%%%%%%%%%%%%%%%%%%%%%%%%%%%%%%%%%%%%%%%%%%%%%%%%%%%%%%%%%%%%%%%%
%
%  result_id=517    51.390 GHz-days 
%  result_id=518   150.772 GHz-days 
%                 -------------------
%                  202.162
%
\begin{table}[htb]
 \caption{Frequency of observed number of real solutions}
 \label{T:intro}

 \begin{tabular}{|l||r|r|r|r|r|r|r||r|}\hline
  \multirow{2}{*}{\textbf{Problem}}&
  \multicolumn{7}{c||}{\textbf{Number of Real Solutions}}&
  \multirow{2}{*}{\textbf{Total}}\\\cline{2-8}
   &$\boldsymbol{0}$&$\boldsymbol{2}$&$\boldsymbol{4}$&$\boldsymbol{6}$&
    $\boldsymbol{8}$&$\boldsymbol{10}$&$\boldsymbol{12}$&\\\hline\hline
  $\TTs^3\cdot\TIs\cdot\Is$\rule{0pt}{13pt}&
       81912&&88738&&7086&&222264&\textbf{400000}\\\hline
  $\ThIIs\cdot\TTs^2\cdot\Is^3$\rule{0pt}{13pt}&
       214375&&231018&&61600&&293007&\textbf{800000}\\\hline
 \end{tabular}
\end{table}
% #585: 311 22^2 1^3: 214375 0 231018 0 61600 0 293007   800000
% #584:  22^3 21 1:    81912 0  88738 0  7086 0 222264   400000
%%%%%%%%%%%%%%%%%%%%%%%%%%%%%%%%%%%%%%%%%%%%%%%%%%%%%%%%%%%%%%%%%%%%%%%%%%%%%%%%%
Together these used $202$ gigahertz-days of computing.
The columns are the number of observed instances with a given number of real solutions.
We only list even numbers, for the number of real solutions is congruent modulo two to the
number of complex solutions.
Empty cells indicate that no instances were observed with that number of real
solutions. 
Notice that the partitions encoding the Schubert problems are symmetric, and that the observed numbers of
real solutions satisfy an additional congruence modulo four.
This congruence occurs for a symmetric Schubert problem in $\Gr(k,2k)$ when the
sum of the lengths of the diagonals of its partitions is at least $k+4$.
This sum is eight for both problems of Table~\ref{T:intro}.

Five problems studied exhibited lower bounds and other structures 
in their observed numbers  of real solutions.  These problems form members of a 
family of Schubert problems, one in each Grassmannian, which we prove 
have this unusual structure on numbers of real solutions.  
We shall show that real solutions to an osculating instance correspond to
real factorizations of an associated polynomial, explaining the observed structures.

This paper is organized as follows.
In Section~\ref{S:definitions} we provide background on the Schubert calculus, the history
of the Shapiro Conjecture, and the work of Eremenko and Gabrielov on topological lower
bounds. 
We then describe the setup, execution, and some of the observations resulting from the
experimental project in Section~\ref{S:experiment}.
In Section~\ref{S:factorization}, we explain the lower bounds and gaps coming from a
family of Schubert problems whose determination we reduce to factoring certain
polynomials.
In Section~\ref{S:further}, we conclude with a discussion of some frequency tables 
exhibiting interesting structure.

%%%%%%%%%%%%%%%%%%%%%%%%%%%%%%%%%%%%%%%%%%%%%%%%%%%%%%%%%%%%%%%%%%%%%%%%%%
%
\section{Background}\label{S:definitions}

We first establish our notation and definitions regarding the osculating Schubert
calculus,  give some additional history of the Shapiro Conjecture, and finally discuss
the topological lower bounds of Eremenko-Gabrielov and Soprunova-Sottile.

%%%%%%%%%%%%%%%%%%%%%%%%%%%%%%%%%%%%%%%%%%%%%%%%%%%%%%%%%%%%%%%%%%%%%%%%%%
%
\subsection{Osculating Schubert calculus}
Let $k<n$ be positive integers.
The Grassmannian \defcolor{$\Gr(k,n)$} (or \defcolor{$\Gr(k,\C^n)$}) is the set of all
$k$-dimensional linear subspaces ($k$-planes) of $\C^n$, which is a complex manifold of
dimension $k(n{-}k)$.
Complex conjugation on $\C^n$ induces a conjugation on $\Gr(k,n)$.
The points of $\Gr(k,n)$ fixed by conjugation are its real points, and they form the
Grassmannian $\Gr(k,\R^n)$ of $k$-planes in $\R^n$. 

The Grassmannian has distinguished Schubert varieties, which are given by the discrete
data of a partition and the continuous data of a flag.
A \demph{partition} is a weakly decreasing sequence 
$\defcolor{\lambda}\colon n{-}k\geq\lambda_1\geq\dotsb\geq\lambda_k\geq 0$ of integers and
a  \demph{flag} is a filtration of $\C^n$:
\[
   \defcolor{\Fdot}\ \colon\  F_1\subsetneq F_2\subsetneq \dotsb \subsetneq F_n=\C^n\,,
\]
where $\dim F_i=i$.
The flag $\Fdot$ is real if $\overline{F_i}=F_i$ for all $i$, so that it is the
complexification of a flag in $\R^n$.
Given a partition $\lambda$ and a flag $\Fdot$, the associated Schubert variety is
 \begin{equation}\label{Eq:SchubVar}
   \defcolor{X_\lambda\Fdot}\ :=\ 
   \{ H\in\Gr(k,n) \,\mid\, \dim H\cap F_{n-k+i-\lambda_i}\geq i
     \ \mbox{ for }i=1,\dotsc,k\}\,.
 \end{equation}
This is an irreducible subvariety of the Grassmannian of codimension
$\defcolor{|\lambda|}:=\lambda_1+\dotsb+\lambda_k$.
From the definition, we see that $\overline{X_\lambda\Fdot}=X_\lambda\overline{\Fdot}$. 

A list $\defcolor{\blambda}=(\lambda^1,\dotsc,\lambda^m)$ of partitions which
satisfies the numerical condition
 \begin{equation}\label{Eq:SchubertProblem}
   |\lambda^1|+|\lambda^2|+\dotsb+|\lambda^m|\ =\ k(n{-}k)
 \end{equation}
is a \demph{Schubert problem}.
Given a Schubert problem $\blambda$ and general flags $\Fdot^1,\dotsc,$ $\Fdot^m$, 
Kleiman's Transversality Theorem~\cite{Kl74} implies that the intersection 
 \begin{equation}\label{Eq:instance}
   X_{\lambda^1}\Fdot^1 \cap
   X_{\lambda^2}\Fdot^2 \cap \dotsb  \cap
   X_{\lambda^m}\Fdot^m
 \end{equation}
is generically transverse.
The numerical condition~\eqref{Eq:SchubertProblem} implies that it is zero-dimensional (or
empty) and therefore consists of finitely many points.
The number of points does not depend upon the choice of general
flags and may be computed using algorithms from the Schubert calculus~\cite{Fu97}.
The intersection~\eqref{Eq:instance} is an \demph{instance} of the Schubert
problem $\blambda$ and its points are the \demph{solutions} to this
instance. 

We will not be concerned with general instances of Schubert problems but rather with
instances given by flags that osculate a common rational normal curve.
Let $\defcolor{\gamma}\colon\C\to\C^n$ be the following parameterized rational normal curve
 \begin{equation}\label{Eq:RNC}
    \gamma(t)\ :=\ \bigl(1\,,\,t\,,\,\tfrac{t^2}{2}\,,\,\tfrac{t^3}{3!}
                      \,,\,\dotsc\,,\,\tfrac{t^{n-1}}{(n-1)!}\bigr)\,.
 \end{equation}
(This choice of $\gamma$ is no restriction as all rational normal curves are projectively
equivalent.)
For each $t\in\C$, the \demph{osculating flag $\Fdot(t)$} has as its $i$-dimensional
subspace the $i$-plane $F_i(t)$ osculating the curve $\gamma$ at $\gamma(t)$:
 \begin{eqnarray}
    \defcolor{F_i(t)}& :=&  \Span\{\gamma(t)\,,\,\gamma'(t)
               \,,\,\dotsc\,,\,\gamma^{(i-1)}(t)\}\nonumber\\
    &=&\rowspace\left(\frac{t^{b-a}}{(b{-}a)!} \label{Eq:matrixF}
         \right)_{\substack{a=1,\dotsc,i\\b=a,\dotsc,n}}\ .
 \end{eqnarray}
(The remaining entries in this matrix are zero.)

An \demph{osculating instance} of a Schubert problem $\blambda$ is one given by osculating
flags, 
 \begin{equation}\label{Eq:osculatingInstance}
   X_{\lambda^1}\Fdot(t_1)\,\cap\,
   X_{\lambda^2}\Fdot(t_2)\,\cap\;\dotsb\;\cap\,
   X_{\lambda^m}\Fdot(t_m)\,.
 \end{equation}
Here, $t_1,\dotsc,t_m$ are distinct points of $\P^1$.
Osculating flags are not general for intersections of Schubert varieties as demonstrated
in~\cite[\S\ 2.3.6]{RSSS}, so Kleiman's Theorem does not imply that the
intersection~\eqref{Eq:osculatingInstance} is transverse.  
However, Eisenbud and Harris~\cite{EH83} noted that if $H\in X_\lambda\Fdot(t_0)$ then
its Wronskian $\defcolor{w_H(t)}$ vanishes to order $|\lambda|$ at $t=t_0$.
As the Wronskian (a form on $\P^1$) has degree $k(n{-}k)$, they deduced 
that~\eqref{Eq:osculatingInstance} is at most zero-dimensional.
Later, Mukhin, Tarasov, and Varchenko~\cite{MTV09b} showed that the intersection is
transverse when $t_1,\dotsc,t_m$ are real (and therefore also when they are general).

Eisenbud and Harris also noted that if $t_0$ is a root of order
$\ell$ of the Wronskian $w_H(t)$ of $H$, then there is a unique partition $\lambda$ with
$|\lambda|=\ell$ such that $H\in X_\lambda\Fdot(t_0)$.
This implies the following partial converse to Schubert
problems. 

%%%%%%%%%%%%%%%%%%%%%%%%%%%%%%%%%%%%%%%%%%%%%%%%%%%%%%%%%%%%%%%%%%%%%%%%%%%%%%%%%
\begin{proposition}\label{P:uniqueness}
  For each $H\in\Gr(k,n)$, there is a unique Schubert problem $\blambda$
  and unique points $t_1,\dotsc,t_m\in\P^1$ for which $H$ lies in the
  intersection~\eqref{Eq:osculatingInstance}. 
\end{proposition}
%%%%%%%%%%%%%%%%%%%%%%%%%%%%%%%%%%%%%%%%%%%%%%%%%%%%%%%%%%%%%%%%%%%%%%%%%%%%%%%%%

To simplify notation, we henceforth write \defcolor{$X_\lambda(t)$} for the Schubert
variety $X_\lambda\Fdot(t)$.

As $\overline{\gamma^{(i)}(t)}=\gamma^{(i)}(\overline{t})$, we have 
$\overline{\Fdot(t)}=\Fdot(\overline{t})$, and therefore 
$\overline{X_\lambda\Fdot(t)}=X_\lambda(\overline{t})$.
A consequence of these observations and Proposition~\ref{P:uniqueness} is the following
corollary.

%%%%%%%%%%%%%%%%%%%%%%%%%%%%%%%%%%%%%%%%%%%%%%%%%%%%%%%%%%%%%%%%%%%%%%%%%%
\begin{corollary}\label{cor:realOscInst}
 Let $\blambda=(\lambda^1,\dotsc,\lambda^m)$ be a Schubert problem and
 $t_1,\dotsc,t_m\in\P^1$ be distinct. 
 The instance
 \[ 
    X_{\lambda^1}(t_1)\,\cap\,
    X_{\lambda^2}(t_2)\,\cap\;\dotsb\; \cap\,
    X_{\lambda^m}(t_m)
 \]
 of the Schubert problem $\blambda$ is a real variety if and only if for each 
 $i=1,\dotsc,m$ there exists $1\leq j\leq m$ such that 
 $\lambda_i=\lambda_j$ and $t_i=\overline{t_j}$.
\end{corollary}
%%%%%%%%%%%%%%%%%%%%%%%%%%%%%%%%%%%%%%%%%%%%%%%%%%%%%%%%%%%%%%%%%%%%%%%%%%

Corollary~\ref{cor:realOscInst} asserts that the obviously sufficient condition for an
osculating instance of a Schubert problem to be a real variety, namely that each 
complex conjugate pair of osculation points have the same Schubert condition, is in fact
necessary. 

%%%%%%%%%%%%%%%%%%%%%%%%%%%%%%%%%%%%%%%%%%%%%%%%%%%%%%%%%%%%%%%%%%%%%%%%%%
\subsection{The Shapiro Conjecture and its generalizations}

One motivation for studying real osculating instances of Schubert problems is
the conjecture of Shapiro and Shapiro, which was given two different proofs by Mukhin,
Tarasov, and Varchenko.

%%%%%%%%%%%%%%%%%%%%%%%%%%%%%%%%%%%%%%%%%%%%%%%%%%%%%%%%%%%%%%%%%%%%%%%%%%%%%%%%%
\begin{theorem}[\cite{MTV09,MTV09b}]
  Given any osculating instance of a Schubert problem
\[
   X_{\lambda^1}\Fdot(t_1)\,\cap\;
   X_{\lambda^2}\Fdot(t_2)\,\cap\;\dotsb\;\cap\,
   X_{\lambda^m}\Fdot(t_m)\,,
\]
 in which $t_1,\dotsc,t_m\in\R\P^1$, the intersection is transverse with all
 of its points real.
\end{theorem}
%%%%%%%%%%%%%%%%%%%%%%%%%%%%%%%%%%%%%%%%%%%%%%%%%%%%%%%%%%%%%%%%%%%%%%%%%%%%%%%%%

The Shapiro Conjecture was made by the brothers Boris and Michael Shapiro in 1993, and 
popularized through significant computer experimentation and partial
results~\cite{Sot00,Ve00}. 
An asymptotic version (where the points $t_i$ are sufficiently clustered and all except
two of the $\lambda^i$ consist of one part) was proven in~\cite{So99}.
The first breakthrough was given by Eremenko and Gabrielov~\cite{EG02b} who 
used complex analysis to prove it when $\min\{k,n{-}k\}=2$.
In this case it is equivalent to
the statement that a rational function whose critical points lie on a circle in $\P^1$
maps that circle to a circle.
Later, Mukhin, Tarasov, and Varchenko proved the full conjecture~\cite{MTV09,MTV09b} using
methods from mathematical physics.

While the Shapiro Conjecture may be formulated in any flag manifold, it is false in
general (except for the orthogonal Grassmannian~\cite{Purbhoo}).
Significant experimental work has uncovered the limits of its validity, as well as
generalizations and extensions that are likely true~\cite{secant,monotone,RSSS}, and has
led to a proof of one generalization (the Monotone Conjecture) in a special
case~\cite{EGS06}. 
For a complete account, see~\cite{FRSC} or~\cite[Chs.~9--14]{IHP}.

%%%%%%%%%%%%%%%%%%%%%%%%%%%%%%%%%%%%%%%%%%%%%%%%%%%%%%%%%%%%%%%%%%%%%%%%%%
%
\subsection{Topological lower bounds}
While studying the Shapiro Conjecture, Eremenko and Gabrielov looked at real
osculating instances of the form
 \begin{equation}\label{Eq:EGlower}
  X_{\Is}(t_1)\,\cap\,
  X_{\Is}(t_2)\,\cap\;\dotsb\;\cap\,
  X_{\Is}(t_m)\,\cap\,
  X_{\lambda}(\infty)\,,
 \end{equation}
where $m+|\lambda|=k(n{-}k)$ and $\{t_1,\dotsc,t_m\}$ is a real set in that 
$\{t_1,\dotsc,t_m\}=\{\overline{t_1},\dotsc,\overline{t_m}\}\subset\P^1$, equivalently,
$w(t):=\prod_i(t-t_i)$ is a real polynomial.
The points in~\eqref{Eq:EGlower} are the fiber of the Wronski map
over the real polynomial $w(t)$ restricted to the Schubert variety
$X_\lambda(\infty)$.
Eremenko and Gabrielov~\cite{EG02} gave a formula for the topological degree of this
Wronski map restricted to the real points of $X_\lambda(\infty)$ (and lifted to an
oriented double cover).
This topological degree is a \demph{topological lower bound} on the number of real points
in the intersection~\eqref{Eq:EGlower}.  
This follows from the formula for the topological degree of a map $f\colon X\to Y$
between oriented manifolds,
\[
    \deg f\ =\ \sum_{x\in f^{-1}(y)} \mbox{sign}(df_x)\,,
\]
where $y\in Y$ is a regular value of $f$ and $\mbox{sign}(df_x)$ is $1$ if the orientation 
of $T_yY$ given by the differential $df_x(T_xX)$ agrees with its orientation from $Y$, and
$-1$ if the orientations do not agree.

This was generalized by Soprunova and Sottile~\cite[Th.~6.4]{SS06} to intersections of the
form
 \begin{equation}\label{Eq:SSlower}
  X_{\mu}(0)\,\cap\,
  X_{\Is}(t_1)\,\cap\,
  X_{\Is}(t_2)\,\cap\;\dotsb\;\cap\,
  X_{\Is}(t_m)\,\cap\,
  X_{\lambda}(\infty)\,,
 \end{equation}
where $m+|\lambda|{+}|\mu|=k(n{-}k)$ and $\{t_1,\dotsc,t_m\}\subset\C^*$ is a real set.
This intersection is again a fiber of the Wronski map restricted to  
$X_{\mu}(0)\cap X_{\lambda}(\infty)$ (and lifted to an oriented double cover).
They expressed the topological degree in terms of sign-imbalance.

A partition $\lambda$ is represented by its Young diagram, which is a left-justified array
of boxes with $\lambda_i$ boxes in row $i$.
When $\mu\subset\lambda$, we have the \demph{skew partition $\lambda/\mu$}, which is
the set-theoretic difference $\lambda\smallsetminus\mu$ of their diagrams.
For example, if
\[
   \lambda\ =\ \raisebox{-5pt}{\includegraphics{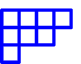}}
   \qquad\mbox{and}\qquad
   \mu\ =\ \raisebox{-2.5pt}{\includegraphics{pictures/21.eps}}
   \qquad\mbox{then}\qquad
   \lambda/\mu\ =\ \raisebox{-5pt}{\includegraphics{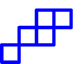}}\ .
\]

Given $\lambda$, let \defcolor{$\lambda^c$} be the partition 
$n{-}k{-}\lambda_k\geq\dotsb\geq n{-}k{-}\lambda_1$, the difference between the
$k\times(n{-}k)$ rectangle and $\lambda$.
For example, if $k=3$, $n=7$, and $\lambda=(3,0,0)$, then $\lambda^c=(4,4,1)$.
A \demph{Young tableau} of shape $\lambda/\mu$ is a filling of the boxes in $\lambda/\mu$
with the consecutive integers $1,2,\dotsc, |\lambda|{-}|\mu|$ which increases across each row
and down each column.
The \demph{standard filling} is the tableau whose numbers are in reading order.
Here are four tableaux of shape $(4,4,1)/(1)$.
The first has the standard filling.
\[
  \Skew{1}{2}{3}{4}{5}{6}{7}{8}
   \qquad
  \Skew{3}{5}{7}{1}{4}{6}{8}{2}
   \qquad
  \Skew{2}{3}{5}{1}{4}{6}{7}{8}
   \qquad
  \Skew{1}{3}{6}{2}{4}{5}{8}{7}
\]
Let \defcolor{$Y(\lambda/\mu)$} be the set of Young tableaux of shape $\lambda/\mu$.
Each tableau $T$ has a parity, $\defcolor{\mbox{sign}(T)}\in\{\pm1\}$, which is the sign
of the permutation mapping the standard filling to $T$.
The \demph{sign-imbalance} of $\lambda/\mu$ is 
\[
   \defcolor{\sigma(\lambda/\mu)}\ :=\ 
    \Bigl|\sum_{T\in Y(\lambda/\mu)} \mbox{sign}(T)\Bigr|\,.
\]

Algorithms in the Schubert calculus~\cite{Fu97} imply that the number of complex points in
the intersection~\eqref{Eq:SSlower} is the number $|Y(\lambda^c/\mu)|$ of tableaux of
shape $\lambda^c/\mu$.
Soprunova and Sottile show that the topological degree of the appropriate Wronski map is
the sign-imbalance of $\lambda^c/\mu$. 
We deduce the following proposition.

%%%%%%%%%%%%%%%%%%%%%%%%%%%%%%%%%%%%%%%%%%%%%%%%%%%%%%%%%%%%%%%%%%%%%%%%%%%%%%%%%
\begin{proposition}[\cite{EG02,SS06}]\label{P:topological_lower}
  If $\{t_1,\dotsc,t_m\}\subset \C^*$ is a real set, then the number of real points in the
  intersection~$\eqref{Eq:SSlower}$ is at least the sign-imbalance $\sigma(\lambda^c/\mu)$
  of $\lambda^c/\mu$.
\end{proposition}
%%%%%%%%%%%%%%%%%%%%%%%%%%%%%%%%%%%%%%%%%%%%%%%%%%%%%%%%%%%%%%%%%%%%%%%%%%%%%%%%%

When $\lambda=\mu=\emptyset$, Eremenko and Gabrielov gave a closed formula for this
topological lower bound, which showed that it is strictly positive when $n$ is odd and 
zero when $n$ is even~\cite{EG02}.
Later they showed that if both $n$ and $k$ are even, there is an
intersection~\eqref{Eq:EGlower} with no real points~\cite{EG03}, showing in these cases
that the topological lower bound is sharp.
Investigating when the topological lower bounds of Proposition~\ref{P:topological_lower}
are sharp and  when they are not was a focus of the experiment.

%%%%%%%%%%%%%%%%%%%%%%%%%%%%%%%%%%%%%%%%%%%%%%%%%%%%%%%%%%%%%%%%%%%%%%%%%%
%
\section{Experimental project}\label{S:experiment}

We describe a large computational
experiment to study structure in the number of real solutions to real instances of
osculating Schubert problems.
The data for this experiment kept track of which pairs of the osculating flags were complex
conjugate, as preliminary computations and the Mukhin-Tarasov-Varchenko Theorem showed that
this affected the numbers of real solutions. 
The computations were carried out symbolically in exact arithmetic, with real osculating
instances of Schubert problems being formulated as systems of polynomial equations.
We sketch the execution of the experiment and then close with a discussion of some of the
data gathered, which is available to browse online~\cite{Lower_Exp}.

%%%%%%%%%%%%%%%%%%%%%%%%%%%%%%%%%%%%%%%%%%%%%%%%%%%%%%%%%%%%%%%%%%%%%%%%%%%%%%%%%
%
\subsection{Osculation type}
We have expressed Schubert problems $\blambda$ as lists of partitions.
Also useful and more compact is multiplicative notation.
For example, the Schubert problem $\blambda=(\TIs,\TIs,\I,\I,\I)$ in $\Gr(3,6)$ with
six solutions is written multiplicatively as $\TIs^2\cdot \I^3$ or as
$\TIs^2\cdot \I^3 = 6$, when we wish to give its number of solutions.

The topological lower bound for the Schubert problem $\TTs\cdot \I^5 = 6$ in $\Gr(3,6)$ is
the sign imbalance $\sigma(\TT^c)=\sigma(\ThIIs) = 2$. 
Thus, if the instance 
 \begin{equation}\label{Eq:refine}
  X_{\TTs}(0)\,\cap\,
  X_{\Is}(t_1)\,\cap\;\cdots\;\cap\,
  X_{\Is}(t_5)
 \end{equation}
of the Schubert problem $\TTs\cdot \I^5 = 6$ has 
$\{\overline{t_1},\dotsc,\overline{t_5}\}=\{t_1,\dotsc,t_5\}$, then~\eqref{Eq:refine}
contains at least two real points. 
If $t_1,\dotsc,t_5\in\R\P^1$ then all six points in~\eqref{Eq:refine} are real by the
Mukhin-Tarasov-Varchenko Theorem. 
This illustrates that the lower bound on the number of real solutions to an osculating
instance of a Schubert problem is sensitive to the number of real osculation points. 
Given a Schubert problem 
$(\lambda^1)^{a_1}\dotsb(\lambda^m)^{a_m}$ and a corresponding real osculating instance
$X$, the \demph{osculation type $r$} of $X$ is the list
$r=(r_{\lambda^1},\dotsc,r_{\lambda^m})$ where $r_{\lambda^i}$ is the number of Schubert 
varieties of the form $X_{\lambda^i}(t)$ containing $X$ with $t$ real. 
Since $X$ is real, we have $r_{\lambda^i}\equiv a_i \mod 2$ for each $i$.

Table~\ref{Ta:5.1e7=6} is from the experiment.
It records how often a given
number of real solutions was observed for a given osculation type in 400000 random real
instances of the Schubert problem $\V\cdot\I^7 = 6$.  
%%%%%%%%%%%%%%%%%%%%%%%%%%%%%%%%%%%%%%%%%%%%%%%%%%%%%%%%%%%%%%%%%%%%%%%%%%%%%%%%%
\begin{table}[htb]
 \caption{Frequency table for $\V\cdot\I^7 = 6$ in $\Gr(2,8)$}
  \label{Ta:5.1e7=6}

 \begin{tabular}{|c||r|r|r|r||r|}\hline
   \multirow{2}{*}{$r_{\Is}$}&
   \multicolumn{4}{c||}{\textbf{Number of Real Solutions}}&
   \multirow{2}{*}{\textbf{Total}}\\\cline{2-5}
   &$\boldsymbol{0}$&$\boldsymbol{2}$&$\boldsymbol{4}$&$\boldsymbol{6}$&\\\hline\hline
   \textbf{7}&&&&100000&\textbf{100000}\\\hline
   \textbf{5}&&&77134&22866&\textbf{100000}\\\hline
   \textbf{3}&&47138&47044&5818&\textbf{100000}\\\hline
   \textbf{1}&8964&67581&22105&1350&\textbf{100000}\\\hline %\hline
%   \textbf{Total}&\textbf{8964}&\textbf{114719}&\textbf{146283}&
%    \textbf{130034}&\textbf{400000}\\\hline
 \end{tabular}
\end{table}
%%%%%%%%%%%%%%%%%%%%%%%%%%%%%%%%%%%%%%%%%%%%%%%%%%%%%%%%%%%%%%%%%%%%%%%%%%%%%%%%%
In every computed instance when $r_{\Is} = 7$, all six solutions were real,
agreeing with the Mukhin-Tarasov-Varchenko Theorem.
The nonzero entry $8964$ in the bottom row with $r_{\Is} = 1$ indicates the 
sharpness of the topological lower bound $\sigma(\V^c)=0$ for $\V\cdot \I^7$.  
The table suggests the lower bound of $r_{\Is}{-}1$ for the number of real solutions to an
instance of this Schubert problem, which we prove in Section~\ref{S:factorization}. 
Studying this Schubert problem used $1.8$ gigahertz-days of computing.

%%%%%%%%%%%%%%%%%%%%%%%%%%%%%%%%%%%%%%%%%%%%%%%%%%%%%%%%%%%%%%%%%%%%%%%%%%%%%%%%%
\subsection{Schubert problems in local coordinates}

The fundamental fact that underlies this experiment is that we may represent
Schubert problems on a computer through systems of equations, from which we may extract 
the number of real solutions.
We explain how to formulate Schubert problems as systems of equations.

A point in the Grassmannian $\Gr(k,n)$ is the row space of a matrix $M\in \Mat_{k\times n}$.
Thus the set of $k\times n$ matrices with complex entries parameterizes $\Gr(k,n)$ 
via the map 
 \begin{eqnarray*}
  \Mat_{k\times n} & \longrightarrow &\Gr(k,n)\,,\\
  M &\longmapsto& \rowspace (M)\,.
 \end{eqnarray*}
This restricts to an injective map from $\Mat_{k\times (n-k)}$ to a dense open set of
$\Gr(k,n)$, 
 \begin{eqnarray*}
  \Mat_{k\times (n-k)} & \longrightarrow& \Gr(k,n)\,,\\
  M &\longmapsto& \rowspace (\Id_{k\times k}:M)\,,
 \end{eqnarray*}
giving local coordinates for the Grassmannian. 

Schubert varieties also have local coordinates.
The flag $\Fdot(\infty)$ has $i$-dimensional subspace
\[
    F_i\ =\ \Span\{e_{n+1-i},\dotsc,e_{n-1},e_n\}\,,
\]
where $\{e_1,\dotsc,e_n\}$ are the standard basis vectors of $\C^n$.
The Schubert variety $X_\lambda(\infty)$ has local coordinates given given by
matrices $M$ whose entries satisfy 
\[
    M_{i,j}\ =\ \left\{
    \begin{array}{rcl}
     1 &\ &\mbox{if } j=i+\lambda_{k+1-i}\,,\\
     0 &\ &\mbox{if } j=a+\lambda_{k+1-a}\mbox{ for }a\neq i\,,\\
     0 &\ &\mbox{if } j<i+\lambda_{k+1-i}\,, 
    \end{array}\right.
\]
and whose other entries are arbitrary complex numbers.
For example, $X_{\TIs}(\infty)\subset \Gr(3,6)$ is parameterized by matrices of the form
\[
 \left(
  \begin{matrix}
   1 & M_{1,2} & 0 & M_{1,4} & 0 & M_{1,6} \\
   0 & {\I}    & 1 & M_{2,4} & 0 & M_{2,6} \\
   0 & {\I}    & 0 & {\I}    & 1 & M_{3,6}
  \end{matrix}
 \right)\,.
\]
Here, $\I$ denotes an entry which is zero. 

The flag $\Fdot(0)$ has $F_i=\Span\{e_1,\dotsc,e_i\}$ and local coordinates for
$X_\lambda(0)$ are given by reversing the rows of those for $X_\lambda(\infty)$.
More interesting is that 
$X_\lambda(\infty)\cap X_\mu(0)$ has local coordinates given by matrices $M$
 whose entries satisfy 
\[
    M_{i,j}\ =\ \left\{
    \begin{array}{rcl}
     1 &\ &\mbox{if } j=i+\lambda_{k+1-i}\,,\\
     0 &\ &\mbox{if } j<i+\lambda_{k+1-i}\,,\\
     0 &\ &\mbox{if } n{+}1-i-\mu_i<j\,,
    \end{array}\right.
\]
and whose other entries are arbitrary complex numbers.
For example, $X_{\IIs}(\infty)\cap X_{\TIgs}(0)\subset \Gr(3,6)$ has local coordinates given by
matrices of the form 
\[
 \left(
  \begin{matrix}
   1 & M_{1,2} & {\Ig} & 0       & {\Ig}   & 0       \\
   0 & {\I}    & 1     & M_{2,4} & {\Ig}   & 0       \\
   0 & {\I}    & 0     & 1       & M_{3,5} & M_{3,6}
  \end{matrix}
 \right)\,.
\]

Let us represent a flag $\Fdot$ by a matrix with rows $f_1,\dotsc,f_n$ so that
\[
   F_i\ =\ \rowspace\left(\begin{array}{c}f_1\\\vdots\\f_i\end{array}\right)\ ,
\]
(and also write $F_i$ for this matrix).
Then the condition~\eqref{Eq:SchubVar} that $H\in X_\nu\Fdot$ is expressed in any of these
local coordinates $M$ for $\Gr(k,n)$, $X_\lambda(\infty)$, or 
$X_\lambda(\infty)\cap X_\mu(0)$ by 
 \begin{equation}\label{Eq:rkEqs}
    \rank\left(\begin{array}{c}M\\F_{n-k+i-\nu_i}\end{array}\right)\ 
    \leq n-\nu_i\ \mbox{ for }i=1,\dotsc,k\,.
 \end{equation}
Each rank condition is given by the vanishing of all minors of the matrix of size
$n{-}\nu_i+1$, and therefore by a system of polynomials in the entries of $M$.

Given an osculating instance of a Schubert problem $\blambda$,
 \[ 
    X_{\lambda^1}(t_1)\,\cap\,
    X_{\lambda^2}(t_2)\,\cap\;\dotsb\; \cap\,
    X_{\lambda^m}(t_m)\,,
 \]
the rank equations~\eqref{Eq:rkEqs} formulate it as a system of polynomials in local
coordinates for the Grassmannian. 
If, say $t_m=\infty$, then we may formulate this instance in the smaller set of local
coordinates for $X_{\lambda^m}(\infty)$, and if we also have $t_{m-1}=0$, a further reduction
is possible using the coordinates for $X_{\lambda^{m-1}}(0)\cap X_{\lambda^m}(\infty)$.

We entertain these possibilities because solving Schubert problems using symbolic
computation is sensitive to the number of variables.
Whenever possible, computations in the experiment assume that $\infty$ and $0$
are osculation points. 
If two points of osculation are real, this is achieved by a simple change of variables.

A further sleight of hand is necessary for this to be
computationally feasible. 
Symbolic computation works best over the field $\Q$, and not as well over $\Q[\sqrt{-1}]$.
We choose our nonreal osculation points to lie in  $\Q[\sqrt{-1}]$, but
exploit that they come in pairs to formulate equations over $\Q$.
Indeed, if $t\not\in\R$, then the polynomials $I$ for $X_\lambda(t)$ described above will have
complex coefficients.
Taking real and imaginary parts of the polynomials in $I$ will give real polynomials that
define $X_\lambda(t)\cap X_\lambda(\overline{t})$.

%%%%%%%%%%%%%%%%%%%%%%%%%%%%%%%%%%%%%%%%%%%%%%%%%%%%%%%%%%%%%%%%%%%%%%%%%%%%%%%%%
\subsection{Methods}

This experiment formulates real osculating instances of Schubert problems,
determines their number of real solutions, records the result by osculation type, and
repeats this hundreds of millions of times on a supercomputer.
The overall framework and basic code was adapted from that developed for other
experimental projects our group has run to study generalizations of the Shapiro
Conjecture~\cite{secant,monotone}.
This experimental design and core code are due to Hillar and are explained in detail
in~\cite{Exp-FRSC}. 

The experiment was organized around a MySQL database hosted at Texas A\&M University.
The database keeps track of all aspects of the experiment, from the problems to be computed (and how
they are computed) to the current state of the computation to the data from the computation.
We wrote web-based tools to communicate with the database and display the data from the
experiment, allowing us to monitor the computations.
The computation was controlled by a perl script that, when run, gets a problem to work on
from the database, sets up and runs the computation, and upon conclusion, updates the
database with its results. 
The perl script may be run on any machine with access to the database, and we used
job-scheduling tools to control its running on the two clusters we have access to at Texas
A\&M. 
These are the brazos cluster in which our research
group controls 20 eight-core nodes, and the  Calclabs, which consists of over 200 Linux
workstations that moonlight as a Beowulf cluster---their day job being calculus instruction. 
In all, the experiment solved over 344 million real osculating instances of 756 Schubert
problems and used 549 gigahertz-years of computing.
%
%   As of 17 August 2013
%

A separate program was used to load problems into the database, which were first screened 
for possible interest and feasibility.
During loading, it was also determined which scheme of local coordinates to use for computing
that particular problem, similar to the protocol followed in~\cite{secant,monotone}.
The actual computation also followed those experiments, and more detail, including
references, is given in {\it loc.\ cit}.
Briefly, an instance was formulated as a system of polynomials in local coordinates; then 
Gr\"obner basis methods implemented in Singular~\cite{Singular} computed a univariate
eliminant 
whose number of real roots is equal to the number of real solutions.
These were counted using the symbolic method of Sturm (implemented in Singular's {\tt
  rootsur}~\cite{rootsur_lib} library).

The experiment was designed to be robust.
All calculations are repeatable as the data are deterministically generated from random
seeds which are stored in the database.
The inevitability of problems, from processor failures to power outages to erroneous human
intervention with the database, motivated us to build in recoverability from all such
events.

%%%%%%%%%%%%%%%%%%%%%%%%%%%%%%%%%%%%%%%%%%%%%%%%%%%%%%%%%%%%%%%%%%%%%%%%%%%%%%%%%
\subsection{Some results}\label{s:Frequency}

The results from computing each Schubert problem are recorded in frequency tables such as
Table~\ref{Ta:5.1e7=6} and may be browsed online at~\cite{Lower_Exp}.
The tables display the number of instances of a given osculation type with a given number
of real solutions. 
We display our data in this manner to clarify the dependence of the number of
real solutions upon  osculation type.
For many of the Schubert problems we studied, there is clearly some structure in the
possible number of real solutions in terms of osculation type, but this behavior is not
uniform across all Schubert problems.

For example, in a few problems there appeared to be an upper bound on the number of real
solutions that depends on osculation type, such as the
problem $\TIs^2\cdot \I^3 = 6$ in $\Gr(3,6)$ of Table~\ref{Ta:21e2.1e3=6}.
%%%%%%%%%%%%%%%%%%%%%%%%%%%%%%%%%%%%%%%%%%%%%%%%%%%%%%%%%%%%%%%%%%%%%%%%%%%%%%%%%
\begin{table}[htb]
 \caption{Frequency table for $\TIs^2 \cdot\I^3 = 6$ in $\Gr(3,6)$}
  \label{Ta:21e2.1e3=6}

 \begin{tabular}{|c|c||r|r|r|r||r|}\hline
   \multirow{2}{*}{$r_{\TIs}$}&
   \multirow{2}{*}{$r_{\Is}$}&
   \multicolumn{4}{c||}{\textbf{Number of Real Solutions}}&
   \multirow{2}{*}{\textbf{Total}}\\\cline{3-6}
   &&$\boldsymbol{0}$&$\boldsymbol{2}$&$\boldsymbol{4}$&$\boldsymbol{6}$&\\\hline\hline
   \textbf{2}&\textbf{3}&&&&100000&\textbf{100000}\\\hline
   \textbf{2}&\textbf{1}&27855&11739&22935&37471&\textbf{100000}\\\hline
   \textbf{0}&\textbf{3}&17424&82576&&&\textbf{100000}\\\hline
   \textbf{0}&\textbf{1}&&100000&&&\textbf{100000}\\\hline %\hline
 \end{tabular}
\end{table}
%%%%%%%%%%%%%%%%%%%%%%%%%%%%%%%%%%%%%%%%%%%%%%%%%%%%%%%%%%%%%%%%%%%%%%%%%%%%%%%%%
This used $1.4$ gigahertz-days of computing.
For the vast majority of the Schubert problems there were instances of every osculation
type with all solutions real, and it was not clear what distinguished this second class of
Schubert problems from the first.
In Section~\ref{S:further} we present tables from several other Schubert problems and
discuss other structures we observe.

Of the 756 Schubert problems studied, 273 had the form~\eqref{Eq:SSlower}
and so had a topological lower bound given by Proposition~\ref{P:topological_lower}.
These included the Wronski maps for $\Gr(2,4)$, $\Gr(2,6)$, and $\Gr(2,8)$ for which
Eremenko and Gabrielov had shown that the lower bound of zero was sharp~\cite{EG03}.
For 264 of the remaining 270 cases, instances were computed showing that this topological
lower bound was sharp.  
There were however six Schubert problems for which the topological lower bounds were not
observed.
These were
\[
  \bigl(\,\raisebox{-2.5pt}{\ThTh},\raisebox{-5pt}{\III},\I^7\bigr)\,,\ 
  \Bigl(\,\raisebox{-5pt}{\TTT},\Th,\I^7\Bigr)\,,\ 
  \bigl(\,\raisebox{-2.5pt}{\ThTh},\raisebox{-2.5pt}{\TT},\I^6\bigr)\,,\ 
  \Bigl(\,\raisebox{-5pt}{\TTT},\raisebox{-2.5pt}{\TT},\I^6\Bigr)\,,\ 
  \Bigl(\,\raisebox{-5pt}{\ThThT},\I^8\Bigr)\,,
\]
all in $\Gr(4,8)$, and $\I^9$ in $\Gr(3,6)$.
These have observed lower bounds of $3,3,2,2,2,2$ and sign-imbalances of $1,1,0,0,0,0$,
respectively. 
There is not yet an explanation for the first four, but the last two are 
symmetric Schubert problems, which were observed to have a
congruence modulo four on their numbers of real solutions.
This congruence gives a lower bound of two for both problems $\ThThTs\cdot\I^8=90$ in
$\Gr(4,8)$ and $\I^9=42$ in $\Gr(3,6)$.

Table~\ref{Ta:1e9=42a} shows the result of computing a million real osculating instances
of the Schubert problem $\I^9 = 42$ in $\Gr(3,6)$, which used $1.07$ gigahertz-years of
computing. 
%%%%%%%%%%%%%%%%%%%%%%%%%%%%%%%%%%%%%%%%%%%%%%%%%%%%%%%%%%%%%%%%%%%%%%%%%%%%%%%%%
\begin{table}[htb]
 \caption{Frequency table for $\I^9 = 42$ in $\Gr(3,6)$}
  \label{Ta:1e9=42a}

 \begin{tabular}{|c||r|r|r|r|r|r|r|r|r|r|r|r}\hline
   \multirow{2}{*}{$r_{\Is}$}&
   \multicolumn{11}{c}{\textbf{Number of Real Solutions}}\\\cline{2-13}
   &$\boldsymbol{0}$&$\boldsymbol{2}$&$\boldsymbol{4}$&$\boldsymbol{6}$&$\boldsymbol{8}$&
   $\boldsymbol{10}$&$\boldsymbol{12}$&$\boldsymbol{14}$&$\boldsymbol{16}$&$\boldsymbol{18}$&
   $\boldsymbol{20}$&$\cdots$\\\hline\hline
   \textbf{9}&& && && && && &&$\cdots$\\\hline
   \textbf{7}&&1843&&13286&&69319&&18045&&13998&&$\cdots$\\\hline
   \textbf{5}&&30223&&51802&&57040&&17100&&12063&&$\cdots$\\\hline
   \textbf{3}&&34314&&93732&&47142&&10213&&5532&&$\cdots$\\\hline
   \textbf{1}&& &&151847&&35220&&6416&&2931&&$\cdots$\\\hline %\hline
 \end{tabular}\bigskip \newline

 \begin{tabular}{r|r|r|r|r|r|r|r|r|r|r|r||r|}\hline
   \multicolumn{12}{c||}{\textbf{Number of Real Solutions}}&
   \multirow{2}{*}{\textbf{Total}}\\\cline{1-12}
   $\dotsb$&$\boldsymbol{22}$&$\boldsymbol{24}$&$\boldsymbol{26}$&
   $\boldsymbol{28}$&$\boldsymbol{30}$&$\boldsymbol{32}$&
   $\boldsymbol{34}$&$\boldsymbol{36}$&$\boldsymbol{38}$&
   $\boldsymbol{40}$&$\boldsymbol{42}$&\\\hline\hline
   $\dotsb$&&&&&&&&&&&200000&\textbf{200000}\\\hline
   $\dotsb$&22883&&4592&&11603&&3891&&473&&40067&\textbf{200000}\\\hline
   $\dotsb$&15220&&2767&&4634&&2056&&211&&6884&\textbf{200000}\\\hline
   $\dotsb$&5492&&839&&1194&&504&&65&&973&\textbf{200000}\\\hline
   $\dotsb$&2345&&362&&450&&181&&22&&226&\textbf{200000}\\\hline %\hline
 \end{tabular}
\end{table}
%%%%%%%%%%%%%%%%%%%%%%%%%%%%%%%%%%%%%%%%%%%%%%%%%%%%%%%%%%%%%%%%%%%%%%%%%%%%%%%%%
As with Table~\ref{T:intro}, only numbers of real solutions congruent to 42 modulo four
were observed.

The observed congruence modulo four which were
inspired by these computations (and those of the earlier investigation~\cite{Orig_Lower}) were established in~\cite{HSZ13,HSZ_new}. 

A partition $\lambda$ is \demph{symmetric} if it equals its matrix-transpose.
For example, all except the last of the following are symmetric, 
\[
   \I\,,\ 
   \raisebox{-2.5pt}{\TI}\,,\ 
   \raisebox{-2.5pt}{\TT}\,,\ 
   \raisebox{-5pt}{\ThII}\,,\ 
   \raisebox{-5pt}{\ThThT}\,,\ 
   \raisebox{-5pt}{\ThThTh}\,,\ 
   \raisebox{-7.5pt}{\includegraphics{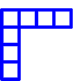}}\,,\ 
   \raisebox{-7.5pt}{\includegraphics{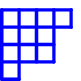}}\,,\ \ 
   \raisebox{-5pt}{\includegraphics{pictures/431.eps}}\,.
\]
A Schubert problem $\blambda$ in $\Gr(k,2k)$ is \demph{symmetric} if every partition in
$\blambda$ is symmetric.
For a symmetric partition $\lambda$, let \defcolor{$\ell(\lambda)$} be the number of boxes
in its main diagonal, which is the maximum number $i$ with $\lambda_i\geq i$.
We state the main result of~\cite{HSZ13,HSZ_new}.

%%%%%%%%%%%%%%%%%%%%%%%%%%%%%%%%%%%%%%%%%%%%%%%%%%%%%%%%%%%%%%%%%%%%%%%%%%%%%%%%%
\begin{proposition}\label{P:mod4}
 Suppose that $\blambda=(\lambda^1,\dotsc,\lambda^m)$ is a symmetric Schubert problem in
 $\Gr(k,2k)$ with $\sum_i \ell(\lambda^i)\geq k{+}4$.
 Then the number of real solutions to a real osculating instance of $\blambda$ is
 congruent to the number of complex solutions modulo four.
\end{proposition}
%%%%%%%%%%%%%%%%%%%%%%%%%%%%%%%%%%%%%%%%%%%%%%%%%%%%%%%%%%%%%%%%%%%%%%%%%%%%%%%%%

One of the symmetric problems, $\ThThThs\cdot \I^7 = 20$ in $\Gr(4,8)$, not only exhibited
this congruence but also appeared to have lower bounds depending upon the osculation
type $r_{\Is}$ as well as  further gaps in its numbers of real solutions (we never observed
12 or 16 real solutions).
Table~\ref{Ta:333.1e7=20} displays the result of computing 400000 real osculating 
%%%%%%%%%%%%%%%%%%%%%%%%%%%%%%%%%%%%%%%%%%%%%%%%%%%%%%%%%%%%%%%%%%%%%%%%%%%%%%%%%
\begin{table}[htb]
 \caption{Gaps and lower bounds for $\ThThThs\cdot \I^7=20$ in $\Gr(4,8)$}
  \label{Ta:333.1e7=20}

 \begin{tabular}{|c||r|c|r|c|r|c|c|c|c|c|r||r|}\hline
   \multirow{2}{*}{$r_{\Is}$}&
   \multicolumn{11}{c||}{\textbf{Number of Real Solutions}}&
   \multirow{2}{*}{\textbf{Total}}\\\cline{2-12}
   &$\boldsymbol{0}$&$\boldsymbol{2}$&$\boldsymbol{4}$&$\boldsymbol{6}$
   &$\boldsymbol{8}$&$\boldsymbol{10}$&$\boldsymbol{12}$&$\boldsymbol{14}$
   &$\boldsymbol{16}$&$\boldsymbol{18}$&$\boldsymbol{20}$&\\\hline\hline
   \textbf{7}&&&&&&&&&&&100000&\textbf{100000}\\\hline
   \textbf{5}&&&&&85080&&&&&&14920&\textbf{100000}\\\hline
   \textbf{3}&&&66825&&30232&&&&&&2943&\textbf{100000}\\\hline
   \textbf{1}&37074&&47271&&14517&&&&&&1138&\textbf{100000}\\\hline %\hline
 \end{tabular}
\end{table}
%%%%%%%%%%%%%%%%%%%%%%%%%%%%%%%%%%%%%%%%%%%%%%%%%%%%%%%%%%%%%%%%%%%%%%%%%%%%%%%%%
instances of this Schubert problem, which used $2.06$ gigahertz-days of computing.
This is a member of a family of Schubert problems
(the problem of Table~\ref{Ta:5.1e7=6} is another) that we can solve completely, and
whose numbers of real solutions have a lower bound depending on osculation type, as well
as gaps.
We explain this in the next section.

%%%%%%%%%%%%%%%%%%%%%%%%%%%%%%%%%%%%%%%%%%%%%%%%%%%%%%%%%%%%%%%%%%%%%%%%%%
%
\section{Lower bounds via factorization}\label{S:factorization}

In our experimentation, we saw that Schubert problems related to 
$\V\cdot\I^7=6$ in $\Gr(2,8)$ (Table~\ref{Ta:5.1e7=6}) and $\ThThThs\cdot \I^7=20$ in
$\Gr(4,8)$ (Table~\ref{Ta:333.1e7=20}) appeared to have gaps and lower bounds depending on
$r_{\Is}$ in their numbers of real solutions.
These are members of a family of Schubert problems, one for each Grassmannian 
$\Gr(k,n)$ with $2\leq k, n{-}k$, which we are able to solve completely, thereby
determining all possibilities for the number of real solutions and explaining these gaps and
lower bounds. 

For $k,n$, let $\minusIt_{k,n}$ ($\minusIt$ for short) denote the partition
$((n{-}k{-}1)^{k-1},0)$ ($n{-}k{-}1$ repeated $k{-}1$ times), which is the
complement of a full hook, $(n{-}k,1^{k-1})$.
For example, 
\[
   \minusIt_{2,6}\ =\ \Th \,,\,\qquad
   \minusIt_{3,8}\ =\ \raisebox{-2.5pt}{\FF}\,,\,\qquad\mbox{and}\qquad
   \minusIt_{4,8}\ =\ \raisebox{-5pt}{\ThThTh}\,.
\]
The osculating Schubert problems in this family all have the form
$\blambda = (\minusIt,\I^{n-1})$ in $\Gr(k,n)$, and they all have topological lower bounds 
$\sigma(\minusIt_{k,n})$ coming from Proposition~\ref{P:topological_lower}.
The multinomial coefficient $\binom{n}{a,b}$ is zero unless $n=a{+}b$, and in that case it
equals $\frac{n!}{a!b!}$.

%%%%%%%%%%%%%%%%%%%%%%%%%%%%%%%%%%%%%%%%%%%%%%%%%%%%%%%%%%%%%%%%%%%%%%%%%%%%%%%%%
\begin{lemma}\label{L:couting_factorizations}
 The Schubert problem $\blambda = (\minusIt,\I^{n-1})$ in $\Gr(k,n)$ has 
 $\binom{n-2}{k-1}$ solutions and 
 $\sigma(\minusIt)= \binom{\lfloor\frac{n-2}{2}\rfloor}%
  {\lfloor\frac{k-1}{2}\rfloor,\lfloor\frac{n-k-1}{2}\rfloor}$, 
 which is zero unless $n$ is even and $k$ is odd.  
\end{lemma}
%%%%%%%%%%%%%%%%%%%%%%%%%%%%%%%%%%%%%%%%%%%%%%%%%%%%%%%%%%%%%%%%%%%%%%%%%%%%%%%%%

%%%%%%%%%%%%%%%%%%%%%%%%%%%%%%%%%%%%%%%%%%%%%%%%%%%%%%%%%%%%%%%%%%%%%%%%%%%%%%%%%
\begin{proof}
 The number of solutions of the Schubert problem $\blambda=(\minusIt,\I^{n-1})$ in
 $\Gr(k,n)$ is the number of Young tableaux of shape $\minusIt^c$, which is a
 full hook $(n{-}k,1^{k-1})$ consisting of one row of length $n{-}k$ and one column of
 length $k$.
 Here are full hooks for $(k,n)$ equal to $(2,6)$, $(3,8)$, $(4,8)$, and $(4,10)$.
\[
  \includegraphics{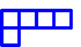}\qquad
  \raisebox{-2.5pt}{\includegraphics{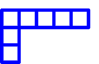}}\qquad
  \raisebox{-5pt}{\includegraphics{pictures/hook4.8.eps}}\qquad
  \raisebox{-5pt}{\includegraphics{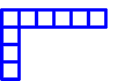}}
\]
 A tableau with such a hook shape has a 1 in its upper left box and the numbers
 $2,\dotsc,n{-}1$ filling out its first row and first column.
 This filling is determined by the $k{-}1$ numbers in the rest of its first
 column.
 Thus there are $\binom{n-2}{k-1}$ tableaux of hook shape  $(n{-}k,1^{k-1})$. 

 Reading a tableau $T$ with hook shape gives a word of the form $1AB$, where $1A$ is the
 first row and $1B$ is the first column.
 This is the permutation corresponding to $T$ whose sign contributes to the
 sign-imbalance. 
 The subwords $AB$ are shuffles of the numbers $\{2,\dotsc,n{-}1\}$.
 Counting these permutations by their lengths is the evaluation of the Gaussian polynomial 
 $\binom{n{-}2}{k{-}1}_q$ at $q=-1$.
 Thus the sign imbalance is $\binom{n{-}2}{k{-}1}_{-1}$, which is well-known
 (see e.g.~\cite[Prop.~7.10]{SS06}) to be 
 $\binom{\lfloor\frac{n-2}{2}\rfloor}{\lfloor\frac{k-1}{2}\rfloor,\lfloor\frac{n-k-1}{2}\rfloor}$.
\end{proof}
%%%%%%%%%%%%%%%%%%%%%%%%%%%%%%%%%%%%%%%%%%%%%%%%%%%%%%%%%%%%%%%%%%%%%%%%%%%%%%%%%

%%%%%%%%%%%%%%%%%%%%%%%%%%%%%%%%%%%%%%%%%%%%%%%%%%%%%%%%%%%%%%%%%%%%%%%%%%%%%%%%%
\begin{theorem}\label{Th:factorization}
 For any $k,n$, the solutions to the osculating instance of the Schubert problem
 $(\minusIt,\I^{n-1})$ in $\Gr(k,n)$,
 \begin{equation}\label{Eq:osc_inst}
   X_{\Is}(t_1)\,\cap\,
   X_{\Is}(t_2)\,\cap\;\dotsb\;\cap\,
   X_{\Is}(t_{n-1})\,\cap\,
   X_{\minusIs}(\infty),
 \end{equation}
 may be identified with all ways of factoring $f'(t)=g(t)h(t)$ where
 \begin{equation}\label{Eq:eff}
   f(t)\ =\ \prod_{i=1}^{n-1} (t-t_i)
 \end{equation}
 with $\deg g=k{-}1$ and $\deg h = n{-}k{-}1$ are monic.
\end{theorem}
%%%%%%%%%%%%%%%%%%%%%%%%%%%%%%%%%%%%%%%%%%%%%%%%%%%%%%%%%%%%%%%%%%%%%%%%%%%%%%%%%

By this theorem, the number of real solutions to a real osculating instance of the
Schubert problem $(\minusIt,\I^{n-1})$ with osculation type $r_{\Is}$ will be the number of
real factorizations $f'(t)=g(t)h(t)$ where $f(t)$ has exactly $r_{\Is}$ real roots,
$\deg g=n{-}k{-}1$, and $\deg h=k{-}1$.
This counting problem was studied in~\cite[Sect.~7]{SS06}, and we recount it here.
Let $r$ be the number of real roots of $f'(t)$.
By Rolle's Theorem, $r_{\Is}-1\leq r\leq n{-}2$.
Then the number \defcolor{$\nu(k,n,r)$} of such factorizations is the coefficient of
$x^{n-k-1}y^{k-1}$ in $(x+y)^r(x^2+y^2)^c$, where $c=\frac{n-2-r}{2}$ is the number of
irreducible quadratic factors of $f'(t)$.

%%%%%%%%%%%%%%%%%%%%%%%%%%%%%%%%%%%%%%%%%%%%%%%%%%%%%%%%%%%%%%%%%%%%%%%%%%%%%%%%%
\begin{corollary}\label{C:factoring}
 The number of real solutions to a real osculating instance of the Schubert problem
 $(\minusIt,\I^{n-1})$~\eqref{Eq:osc_inst} with osculation type $r_{\Is}$ is $\nu(k,n,r)$,
 where $r$ is the number of real roots of $f'(t)$, where $f$ is the
 polynomial~\eqref{Eq:eff}. 
\end{corollary}
%%%%%%%%%%%%%%%%%%%%%%%%%%%%%%%%%%%%%%%%%%%%%%%%%%%%%%%%%%%%%%%%%%%%%%%%%%%%%%%%%

%%%%%%%%%%%%%%%%%%%%%%%%%%%%%%%%%%%%%%%%%%%%%%%%%%%%%%%%%%%%%%%%%%%%%%%%%%%%%%%%%
\begin{remark}
 When $r<n{-}2$, we have that $\nu(k,n,r)\leq\nu(k,n,r{+}2)$, so 
 $\nu(k,n,r_{\Is}{-}1)$ is the lower bound for the number of real solutions to a real
 osculating instance of $(\minusIt,\I^{n-1})$ of osculation type $r_{\Is}$.
 Since at most $\lfloor\frac{n}{2}\rfloor$ different values of $r$ may occur for the
 numbers of real roots of $f'(t)$, the number $\nu(k,n,r)$ satisfies
\[
  \binom{\lfloor\frac{n-2}{2}\rfloor}{\lfloor\frac{k-1}{2}\rfloor,\lfloor\frac{n-k-1}{2}\rfloor}
    \ \leq\ \nu(k,n,r)\ \leq\ 
  \binom{n{-}2}{k{-}1}.\ 
\]
 There will in general be lacunae in the numbers of real solutions, as we saw in
 Table~\ref{Ta:333.1e7=20}. 
 For example, the values of $\nu(5,13,r)$  are
 \begin{equation}\label{Eq:5_13}
   10\,,\ 18\,,\ 38\,,\ 78\,,\ 162\,,\ \mbox{and}\ 330\,.
 \end{equation}
\end{remark}
%%%%%%%%%%%%%%%%%%%%%%%%%%%%%%%%%%%%%%%%%%%%%%%%%%%%%%%%%%%%%%%%%%%%%%%%%%%%%%%%%

%%%%%%%%%%%%%%%%%%%%%%%%%%%%%%%%%%%%%%%%%%%%%%%%%%%%%%%%%%%%%%%%%%%%%%%%%%%%%%%%%
\begin{remark}
 The Schubert problem $(\minusIt,\I^{2k-1})$ in $\Gr(k,2k)$ is symmetric.
 When $k>2$ it
 satisfies the hypotheses of Proposition~\ref{P:mod4} and so its numbers of real
 solutions (the numbers $\nu(k,2k,r)$) are congruent to $\binom{2k-2}{k-1}$
 modulo four.  We deduce this congruence modulo four from Theorem~\ref{Th:factorization} 
 by proving that the number of nonreal solutions to a real osculating instance
 of such a Schubert problem is a multiple of four.
 Equivalently, given a real polynomial $\phi(t)$ of degree $2m=2k{-}2$ with distinct
 roots, the number of nonreal ordered pairs $(g(t),h(t))$ of polynomials of degree $m$
 with $\phi(t)=g(t)h(t)$ is a multiple of four.

 An ordered pair $(g(t),h(t))$ of polynomials of degree $m$ with
 $\phi(t)=g(t)h(t)$ is an \demph{ordered factorization} of $\phi(t)$.
 Given a factorization $\phi(t)=g(t)h(t)$, we have $\phi(t)=h(t)g(t)$, and so 
 $(g(t),h(t))$ and $(h(t),g(t))$ are distinct ordered factorizations of $\phi(t)$.  
 If $g(t)$ (and hence $h(t)$) is not real, and we do not have $\overline{g(t)}=h(t)$, then 
\[
   (g(t),h(t))\,,\ (h(t),g(t))\,,\ 
   (\overline{g(t)},\overline{h(t)})\,,\ (\overline{h(t)},\overline{g(t)})
\]
 are four distinct nonreal ordered factorizations.

 To show that the set of nonreal ordered factorizations of $\phi(t)$ is divisible by four,
 we need only to show that the number for which $\overline{g(t)}=h(t)$ is a multiple of
 four. 
 These can only occur when $\varphi(t)$ has no real roots, for $g(t)$ must have one root 
 from each complex conjugate pair of roots of $\varphi(t)$.
 There are $2^m$ such pairs and so $2^m$ such factorizations, which is a multiple of four
 when $m>1$ and thus when $k>2$.
 This establishes the congruence modulo four of Proposition~\ref{P:mod4} for osculating
 instances of Schubert problems $(\minusIt,\I^{2k-1})$ in $\Gr(k,2k)$ when $k>2$.

 It is interesting to note that by~\eqref{Eq:5_13}, the number of real solutions to this
 problem in $\Gr(5,13)$ also satisfies a congruence modulo four.
\end{remark}
%%%%%%%%%%%%%%%%%%%%%%%%%%%%%%%%%%%%%%%%%%%%%%%%%%%%%%%%%%%%%%%%%%%%%%%%%%%%%%%%%

%%%%%%%%%%%%%%%%%%%%%%%%%%%%%%%%%%%%%%%%%%%%%%%%%%%%%%%%%%%%%%%%%%%%%%%%%%%%%%%%%
\begin{proof}[Proof of Theorem~$\ref{Th:factorization}$]
 The Schubert variety $X_{\minusIs}(\infty)$ consists of those $k$-planes $H$ with
\[
   \dim H \cap F_{i+1}(\infty)\ \geq\ i\qquad\mbox{for } i=1,\dotsc,k{-}1\,.
\]
 By Proposition~\ref{P:uniqueness}, the solutions to~\eqref{Eq:osc_inst} will be points in 
 $X_{\minusIs}(\infty)$ that do not lie in any other smaller Schubert variety
 $X_\lambda(\infty)$.
 This is the Schubert cell of $X_{\minusIs}(\infty)$~\cite{Fu97}, and it consists of the
 $k$-planes $H$ which are row spaces of matrices of the form
\[
  \left(\begin{array}{ccccccccc}
    1&x_1&\dotsb&x_{n-k-1}&x_{n-k}&0&\dotsb&0\\
    0&0  &\dotsb&   0   &  1   &x_{n-k+1}&\dotsb&0\\
   \vdots&\vdots&&\vdots&&\ddots &\ddots&\vdots\\
    0&0  &\dotsb&   0   &  \dotsb&0 & 1& x_{n-1}
    \end{array}\right)\ ,
\]
 where $x_1,\dotsc,x_{n-1}$ are indeterminates.
 If $x_{n-k}=0$, then $H\in X_{\Is}(0)$, but if one of $x_{n-k+1},\dotsc,x_{n-1}$
 vanishes, then $H\in X_{\IIs}(0)$, which cannot occur for a solution
 to~\eqref{Eq:osc_inst}, again by Proposition~\ref{P:uniqueness}.

 We use a slight change from these coordinates.
 Define constants  $\defcolor{g_{n-k-1}}:=1=:\defcolor{h_{k-1}}$  and  
 $\defcolor{c_i}:=(-1)^{n-k-i+1}(n{-}k{-}i)!$ 
 and let $(f,g,h)=(f_0,\,g_0,\dotsc,g_{n-k-1},\,h_0,\dotsc,h_{k-2})$ be variables
 with $h_0,\dotsc,h_{k-2}$ all nonzero.
 If we let \defcolor{$H(f,g,h)$} be the row space of the following matrix (also written $H(f,g,h)$):

 \begin{equation}\label{Eq:coords}
   \left(\begin{array}{ccccccccccc}
     c_1g_{n-k-1} & c_2g_{n-k-2} &  \dotsb & c_{n-k}g_0 &
         \frac{f_0}{h_0} & 0  &\dotsb&0&0\\
     0 &  0  & \dotsb & 0     &-1 &\frac{h_0}{h_1}&\dotsb&0&0\\
     0 &  0  & \dotsb & 0     & 0 &-2 &\ddots&\vdots&\vdots\\
    \vdots&\vdots&&\vdots&\vdots&&\ddots&&0\\
     0 &  0  &  \dotsb & 0     & 0 &\dotsb&-(k{-}2)&\frac{h_{k-3}}{h_{k-2}} &0\\
     0 &  0  &  \dotsb & 0     & 0 &\dotsb&0&-(k{-}1)&\frac{h_{k-2}}{h_{k-1}}\\
   \end{array}\right)\ ,
 \end{equation}
then $H(f,g,h)$ parameterizes the Schubert cell of $X_{\minusIs}(\infty)$.
We postpone the following calculation.

%%%%%%%%%%%%%%%%%%%%%%%%%%%%%%%%%%%%%%%%%%%%%%%%%%%%%%%%%%%%%%%%%%%%%%%%%%%%%%%%%
\begin{lemma}\label{L:factorization}
 The condition for $H(f,g,h)$ to lie in $X_{\Is}(t)$  is
 \begin{equation}\label{Eq:big_det}
  \det \left(\begin{array}{c}H(f,g,h)\\F_{n-k}(t)\end{array}\right)\ 
      \ =\ 
    (-1)^{k(n-k)}\biggl(\, \sum_{i=0}^{n-k-1} \sum_{j=0}^{k-1} 
     \frac{t^{i+j+1}}{i{+}j{+}1} \, g_i\, h_j
      \ +\ f_0 \,\biggr)\,.
 \end{equation}
\end{lemma}
%%%%%%%%%%%%%%%%%%%%%%%%%%%%%%%%%%%%%%%%%%%%%%%%%%%%%%%%%%%%%%%%%%%%%%%%%%%%%%%%%

Call the polynomial in the parentheses $f(t)$.
If $H$ lies in the intersection~\eqref{Eq:osc_inst}, then $(-1)^{k(n-k)}f$ is the
polynomial~\eqref{Eq:eff}. 
If we set
\[
  g(t) :=\ g_0+tg_1+\dotsb+t^{n-k-1}g_{n-k-1}
   \qquad\mbox{and}\qquad
  h(t) :=\ h_0+th_1+\dotsb+t^{k-1}h_{k-1}\,,
\]
then we have $f(0)=f_0$ and $f'(t)=g(t)h(t)$.
Theorem~\ref{Th:factorization} is now immediate.
\end{proof}
%%%%%%%%%%%%%%%%%%%%%%%%%%%%%%%%%%%%%%%%%%%%%%%%%%%%%%%%%%%%%%%%%%%%%%%%%%%%%%%%%

%%%%%%%%%%%%%%%%%%%%%%%%%%%%%%%%%%%%%%%%%%%%%%%%%%%%%%%%%%%%%%%%%%%%%%%%%%%%%%%%%
\begin{proof}[Proof of Lemma~\ref{L:factorization}]
 Expand the determinant~\eqref{Eq:big_det} along the rows of $H(f,g,h)$ 
 to obtain
 \begin{equation}\label{Eq:det_cond}
   \det\left(\begin{array}{c}H(f,g,h)\\ F_{n-k}(t)\end{array}\right)\ =\ 
   \sum_{\alpha\in\binom{[n]}{k}}
    (-1)^{|\alpha|} H(f,g,h)_{\alpha} (F_{n-k}(t))_{\alpha^c}\,,
 \end{equation}
where $\binom{[n]}{k}$ is the collection of subsets of $\{1,\dotsc,n\}$ of
cardinality $k$, $\defcolor{|\alpha|}:=\alpha_1+\dotsb+\alpha_{k}-1-\dotsb-k$,
\defcolor{$H_\alpha$} is the determinant of the $k\times k$ submatrix of $H$ given
by the columns in $\alpha$, and \defcolor{$(F_{n-k}(t))_{\alpha^c}$} is the determinant of the
$(n{-}k)\times(n{-}k)$ submatrix of $F_{n-k}(t)$ formed by the columns in
$\alpha^c:=\{1,\dotsc,n\}\smallsetminus\alpha$. 
These are \demph{minors} of $H(f,g,h)$ and $F_{n-k}(t)$.

A minor $H(f,g,h)_\alpha$ of $H(f,g,h)$ is nonzero only if 
$\alpha=(i,n{-}k{+}1,\dotsc,\widehat{n{-}k{+}j},\dotsc,n)$ for $i\in\{1,\dotsc,n{-}k\}$
and $j\in\{1,\dotsc,k\}$ or $\alpha=(n{-}k{+}1,\dotsc,n{-}1,n)$, the last $k$ columns.
Write \defcolor{$[i,\widehat{j}]$} for the first type and \demph{$[n{-}k]^c$} for the
second. 
Then  $([n{-}k]^c)^c=(1,\dotsc,n{-}k)$ and 
$\defcolor{[\widehat{i},j]}:=[i,\widehat{j}]^c
     =(1,\dotsc,\widehat{i},\dotsc,n{-}k,n{-}k{+}j)$.

With $g_{n-k-1}=1=h_{k-1}$, a calculation shows that
\[
   H_{[i,\widehat{j}]}\ =\ (-1)^{n-k-i+j}(n{-}k{-}i)!\,(j{-}1)!\,g_{n-k-i} \,  h_{j-1}
   \quad\mbox{and}\quad
   H_{[n-k]^c}\ =\ f_0\,.
\]  
Similarly, for any $\alpha\in\binom{[n]}{n-k}$, we use~\eqref{Eq:matrixF} to compute
the minor $(F_{n-k}(t))_\alpha$,
 \begin{multline*}
 \qquad \det \left(\begin{array}{ccc}
    \frac{t^{\alpha_1-1}}{(\alpha_1-1)!}&\dotsb&\frac{t^{\alpha_{n-k}-1}}{(\alpha_{n-k}-1)!}\\
      \vdots&\ddots&\vdots\\
     \frac{t^{\alpha_1-(n-k)}}{(\alpha_1-(n-k))!}&\dotsb&
     \frac{t^{\alpha_{n-k}-(n-k)}}{(\alpha_{n-k}-(n-k))!}\end{array}\right)\\
    \ =\ 
    \frac{t^{|\alpha|}}{(\alpha_1-1)!\dotsb(\alpha_{n-k}-1)!}
      \det \left(\begin{array}{ccc}
     1&\dotsb&1\\
     \alpha_1-1&\dotsb&\alpha_{n-k}-1\\
      \vdots&\ddots&\vdots\\
      (\alpha_1-1)_{n-k-1}&\dotsb& (\alpha_{n-k}-1)_{n-k-1}\end{array}\right)\ ,\qquad\\
    \ =\ 
    \frac{t^{|\alpha|}}{(\alpha_1-1)!\dotsb(\alpha_{n-k}-1)!}
      \det \left(\begin{array}{ccc}
     1&\dotsb&1\\
     \alpha_1&\dotsb&\alpha_{n-k}\\
      \vdots&\ddots&\vdots\\
      \alpha_1^{n-k-1}&\dotsb& \alpha_{n-k}^{n-k-1}\end{array}\right)\ ,\qquad\\
 \end{multline*}
   where $\defcolor{(m)_i}:=m(m-1)\dotsb(m{-}i{+}1)$ and an entry in the first matrix is
   zero if $\alpha_i-j<0$.
   If $\defcolor{\alpha!}:=(\alpha_1-1)!\dotsb(\alpha_{n-k}-1)!$, then
\[
   (F_{n-k}(t))_\alpha\ =\ \frac{t^{|\alpha|}}{\alpha!}\prod_{i<j}(\alpha_j-\alpha_i)
    \ =\ \frac{t^{|\alpha|}}{\alpha!}\Vdm(\alpha)\,,
\]
where \defcolor{$\Vdm(\alpha)$} is the Vandermonde determinant of $\alpha$.
We compute
 \begin{eqnarray*}
  |[i,\widehat{j}]|&=& i+n{-}k{+}1+\dotsb+n - (n{-}k{+}j)-1-\dotsb-k\\ &=&
                         k(n{-}k)-(n{-}k{-}i{+}j)\\
  |[\widehat{i},j]|&=& 1+\dotsb+n{-}k+n{-}k{+}j\ -\ 1-\dotsb-k\ =\ 
                          n{-}k{-}i{+}j\\
   ([\widehat{i},j])!&=& 
         1!\cdot 2!\dotsb(i-2)!i!(i+1)!\dotsb(n{-}k{-}1)!(n{-}k{+}j{-}1)!\\
   \Vdm([\widehat{i},j])&=&  1!\cdot 2!\dotsb(i-2)!\frac{i!}{1}\frac{(i+1)!}{2}\dotsb
                  \frac{(n{-}k{-}1)!}{n{-}k{-}i}\cdot
       \frac{(n{-}k{+}j{-}1)!}{(j-1)!}\cdot\frac{1}{n{-}k{-}i{+}j}\\
   &=& ([\widehat{i},j])!\cdot \frac{1}{(n{-}k{-}i)!(j{-}1)!(n{-}k{-}i{+}j)}\,,
 \end{eqnarray*}
and
\[
   |[n{-}k]^c|\ =\ k(n{-}k)\,,\quad
   [n{-}k]!\ =\ \Vdm([n{-}k])\ =\ 1!\cdot 2! \dotsb (n{-}k{-}1)!\,.
\]
After some cancellation, the determinant~\eqref{Eq:det_cond} becomes 
 \begin{multline*}
   \sum_{i=1}^{n-k} \sum_{j=1}^{k} 
     (-1)^{|[i,\widehat{j}]|} H(f,g,h)_{[i,\widehat{j}]}(F_{n-k}(t))_{[\widehat{i},j]}
       \ \ +\ (-1)^{|[n-k]^c|}H(f,g,h)_{[n-k]^c} (F_{n-k}(t))_{[n-k]}\\
   =\ \ 
    \sum_{i=1}^{n-k} \sum_{j=1}^{k} 
      (-1)^{k(n-k)} \frac{t^{n-k-i+j}}{n{-}k{-}i{+}j} g_{n-k-i} h_{j-1}
      \ \ +\ (-1)^{k(n-k)} f_0\,.
 \end{multline*}
 If we replace $n{-}k{-}i$ by $i$ and $j{-}1$ by $j$ in this sum, we get
\[
   \det\left(\begin{array}{c}H(f,g,h)\\ F_{n-k}(t)\end{array}\right)
   \ =\ 
    (-1)^{k(n-k)}\Bigl(\, \sum_{i=0}^{n-k-1} \sum_{j=0}^{k-1} 
     \frac{t^{i+j+1}}{i{+}j{+}1} \, g_i\, h_j
      \ +\ f_0 \,\Bigr)\,,
\]
which completes the proof.
\end{proof}
%%%%%%%%%%%%%%%%%%%%%%%%%%%%%%%%%%%%%%%%%%%%%%%%%%%%%%%%%%%%%%%%%%%%%%%%%%%%%%%%%

%%%%%%%%%%%%%%%%%%%%%%%%%%%%%%%%%%%%%%%%%%%%%%%%%%%%%%%%%%%%%%%%%%%%%%%%%%
%
\section{More tables from the experiment}\label{S:further}

We present a selection of the tables of real osculating instances of Schubert problems
studied in~\cite{Orig_Lower} and~\cite{Lower_Exp}.
These exhibit intriguing structures in their numbers of real solutions, some of
which we understand, and some that we do not.

\subsection{An enigma.}
Table~\ref{Ta:311.21e3.1e5=54} shows what is perhaps the most complicated structure
we observed.
%%%%%%%%%%%%%%%%%%%%%%%%%%%%%%%%%%%%%%%%%%%%%%%%%%%%%%%%%%%%%%%%%%%%%%%%%%%%%%%%%
% result_id=498
% 24.612 GHz-years.
\begin{table}[htb]
 \caption{Frequency table for 
        $\ThIIb\cdot\TI^3\cdot\I^2 = 54$ in $\Gr(4,8)$}
  \label{Ta:311.21e3.1e5=54}
  %%%%%%%%
  \begin{tabular}{|c|c||r|r|r|r|r|r|r|r|r}\hline
   \multirow{2}{*}{$r_{\TIs}$}&
   \multirow{2}{*}{$r_{\Is}$}&
   \multicolumn{8}{c}{\textbf{Number of Real Solutions}}&\\\cline{3-11}
   &&$\boldsymbol{0}$&$\boldsymbol{2}$&$\boldsymbol{4}$&$\boldsymbol{6}$&
    $\boldsymbol{8}$&$\boldsymbol{10}$&$\boldsymbol{12}$&
    $\boldsymbol{14}$&$\dotsb$\\\hline\hline
   3&2&&&&&&&&&$\dotsb$\\\hline
   3&0&5&2714&13044&111636&59800&88674&20255&52088&$\dotsb$\\\hline
   1&2&81216&235048&72682&109908&9600&52281&2877&12685&$\dotsb$\\\hline
   1&0&&599421&&83350&&53394&&20997&$\dotsb$\\\hline %\hline
  \end{tabular}\bigskip \newline
  %%%%%%%%
  \begin{tabular}{r|r|r|r|r|r|r|r|r|r|r|r}\hline
    \multicolumn{11}{c}{\textbf{Number of Real Solutions}}&\\\hline
   $\dotsb$&$\boldsymbol{16}$&$\boldsymbol{18}$&$\boldsymbol{20}$&$\boldsymbol{22}$&
    $\boldsymbol{24}$&$\boldsymbol{26}$&$\boldsymbol{28}$&$\boldsymbol{30}$&
    $\boldsymbol{32}$&$\boldsymbol{34}$&
     $\dotsb$\\\hline\hline
   $\dotsb$&&&&&&&&&&&$\dotsb$\\\hline
   $\dotsb$&44306&164085&9467&23019&5222&27149&5044&16959&1107&6336&$\dotsb$\\\hline
   $\dotsb$&4953&31084&10&50198&&166418&&&&&$\dotsb$\\\hline
   $\dotsb$&&20896&&16359&&34543&&&&&$\dotsb$\\\hline%\hline
  \end{tabular}\bigskip\newline
  %%%%%%%%
  \begin{tabular}{r|r|r|r|r|r|r|r|r|r|r||r|}\hline
    \multicolumn{11}{c||}{\textbf{Number of Real Solutions}}&
   \multirow{2}{*}{\textbf{Total}}\\\cline{1-11}
    $\dotsb$&$\boldsymbol{36}$&$\boldsymbol{38}$&$\boldsymbol{40}$&$\boldsymbol{42}$&
    $\boldsymbol{44}$&$\boldsymbol{46}$&$\boldsymbol{48}$&
    $\boldsymbol{50}$&$\boldsymbol{52}$&$\boldsymbol{54}$&\\\hline\hline
   $\dotsb$&&&&&&&&&&828960&$\boldsymbol{828960}$\\\hline
   $\dotsb$&1280&15495&1731&13362&240&6292&&35970&1275&102406&$\boldsymbol{828960}$\\\hline
   $\dotsb$&&&&&&&&&&&$\boldsymbol{828960}$\\\hline
   $\dotsb$&&&&&&&&&&&$\boldsymbol{828960}$\\\hline%\hline
  \end{tabular}
  %%%%%%
\end{table}
%%%%%%%%%%%%%%%%%%%%%%%%%%%%%%%%%%%%%%%%%%%%%%%%%%%%%%%%%%%%%%%%%%%%%%%%%%%%%%%%%
This used $24.6$ gigahertz-years of computing.
The first row has only real points of osculation, so the Mukhin-Tarasov-Varchenko Theorem
implies that all 54 solutions are real, as observed. 
All possible numbers of real solutions except 48 were observed for the osculation type of
the  second row.
The third and fourth rows appear to have an upper bound of 26, and the fourth row
exhibits an additional congruence to 54 modulo four.
None of this, besides the first row, is understood.
Compare the upper bound for the last two rows to that in Table~\ref{Ta:21e2.1e3=6} (note
that $2= 2\cdot\lfloor 6/4\rfloor$ and $26= 2\cdot\lfloor 54/4\rfloor$).
A similar structure was also observed for the Schubert problem
$\ThIIs\cdot\TTs\cdot\TIs^2\cdot\I=16$ in $\Gr(4,8)$.

%%%%%%%%%%%%%%%%%%%%%%%%%%%%%%%%%%%%%%%%%%%%%%%%%%%%%%%%%%%%%%%%%%%%%%%%%%%%%%%%%
\subsection{Internal structure.}

Work of Vakil~\cite{Va06b} and others has led to the study of Schubert
problems 
which posses internal structure as encoded by their Galois groups~\cite{Ha79}.
The current state of this investigation is discussed in~\cite[\S~5]{MSJ}.
Intriguingly, in every problem we know whose Galois group is not the full symmetric group,
the internal structure which restricts the Galois group appears to restrict the
possible numbers of real solutions to real osculating instances.

A good example is the Schubert problem 
$\F^2\cdot\IIIs^2\cdot\I^6 = 10$ in $\Gr(4,9)$ of Table~\ref{Ta:4e2111e21e6=10}, which
used $4.14$ gigahertz-years of computing.
The problem is solved by first solving an instance of the problem of four lines ($\I^4=2$)
in a $\Gr(2,4)$ that is given by the four conditions $\F^2\cdot\IIIs^2$.
%%%%%%%%%%%%%%%%%%%%%%%%%%%%%%%%%%%%%%%%%%%%%%%%%%%%%%%%%%%%%%%%%%%%%%%%%%%%%%%%%
% result_id=210
% 4.136 GHz-years
%
\begin{table}[htb]
 \caption{Frequency table for $\F^2\cdot\III^2\cdot\I^6 = 10$ in $\Gr(4,9)$}
  \label{Ta:4e2111e21e6=10}
    \begin{tabular}{|c|c|c||r|r|r|r|r|r||r|}\hline
     \multirow{2}{*}{$r_{\Fs}$}&
     \multirow{2}{*}{$r_{\IIIs}$}&
     \multirow{2}{*}{$r_{\Is}$}&
     \multicolumn{6}{c||}{\textbf{\# Real Solutions}}&
     \multirow{2}{*}{\textbf{Total}}\\\cline{4-9}
      &&&$\boldsymbol{0}$&$\boldsymbol{2}$&$\boldsymbol{4}$&
         $\boldsymbol{6}$&$\boldsymbol{8}$&$\boldsymbol{10}$&\\\hline\hline
     2&2&6&    &    &&    &&8000&\textbf{8000}\\\hline
     2&0&6&5419&    &&    &&2581&\textbf{8000}\\\hline
     0&2&6&2586&    &&    &&5414&\textbf{8000}\\\hline
     0&0&6&    &    &&    &&8000&\textbf{8000}\\\hline
     2&2&4&    &2971&&2202&&2827&\textbf{8000}\\\hline
     2&0&4&5508& 876&& 722&& 894&\textbf{8000}\\\hline
     0&2&4&2469&2051&&1527&&1953&\textbf{8000}\\\hline
     0&0&4&    &2941&&2228&&2831&\textbf{8000}\\\hline
     2&2&2&    &3595&&3374&&1031&\textbf{8000}\\\hline
     2&0&2&5535&1090&&1051&& 324&\textbf{8000}\\\hline
     0&2&2&2539&2472&&2254&& 735&\textbf{8000}\\\hline
     0&0&2&    &3572&&3411&&1017&\textbf{8000}\\\hline
     2&2&0&    &    &&7287&& 713&\textbf{8000}\\\hline
     2&0&0&5378&    &&2386&& 236&\textbf{8000}\\\hline
     0&2&0&2619&    &&4917&& 464&\textbf{8000}\\\hline
     0&0&0&    &    &&7333&& 667&\textbf{8000}\\\hline
    \end{tabular}

\end{table}
%%%%%%%%%%%%%%%%%%%%%%%%%%%%%%%%%%%%%%%%%%%%%%%%%%%%%%%%%%%%%%%%%%%%%%%%%%%%%%%%%
Then, for each of the two solutions to that problem, an instance of the Schubert problem
$\I^6=5$ in $\Gr(2,5)$ is solved, to get 10 solutions in all.
The Galois group of this problem permutes each of these blocks of
five solutions for the two Schubert problems $\I^6=5$ of the second step, and thus 
it acts imprimitively.
Further analysis shows that the Galois group is the wreath product $S_5\wr S_2$,
which has order $(5!)^2\cdot 2!=28800$.

Table~\ref{Ta:two_tables} shows the frequency tables for the two auxiliary problems
$\I^4=2$ in $\Gr(2,4)$ and $\I^6=5$ in $\Gr(2,5)$, which used $12.6$ gigahertz-hours of
computing. 
%%%%%%%%%%%%%%%%%%%%%%%%%%%%%%%%%%%%%%%%%%%%%%%%%%%%%%%%%%%%%%%%%%%%%%%%%%%%%%%%%
% result_id=32    and   result_id=220
%  4.098 GHz-hrs         8.556 GHz-hrs.
%
\begin{table}[htb]
 \caption{Frequency tables for $\I^4=2$ and $\I^6=5$.}
  \label{Ta:two_tables}
    \begin{tabular}{|c||r|r||r|}\hline
     $r_{\Is}$&$\boldsymbol{0}$&$\boldsymbol{2}$&\textbf{Total}\\\hline\hline
     4&&100000&\textbf{100000}\\\hline
     2&32412&67588&\textbf{100000}\\\hline
     0&&100000&\textbf{100000}\\\hline
     \multicolumn{4}{c}{\mbox{\qquad}}
    \end{tabular}
   %%%%%%%
    \hspace{1cm}
   %%%%%%%
    \begin{tabular}{|c||r|r|r||r|}\hline
     $r_{\Is}$&$\boldsymbol{1}$&$\boldsymbol{3}$&
              $\boldsymbol{5}$&\textbf{Total}\\\hline\hline
     6&&&100000&100000\\\hline
     4&36970&36970&35314&\textbf{100000}\\\hline
     2&35314&43081&11222&\textbf{100000}\\\hline
     0&&89105&10895&\textbf{100000}\\\hline
    \end{tabular}
\end{table}
%%%%%%%%%%%%%%%%%%%%%%%%%%%%%%%%%%%%%%%%%%%%%%%%%%%%%%%%%%%%%%%%%%%%%%%%%%%%%%%%%
It is fascinating to compare these to Table~\ref{Ta:4e2111e21e6=10}.
First observe that for $\F^2\cdot\IIIs^2\cdot\I^6$ we have no real solutions only when
$r_{\Ths}+r_{\IIIs}=2$, similar to $\I^4$ having no real solutions only when $r_{\Is}=2$.
The remaining columns of Table~\ref{Ta:4e2111e21e6=10} have the same pattern of dependence
on $r_{\Is}$ as do the columns of $\I^6=5$, except that the number of real solutions is
doubled.

\subsection{Problems of the form $(\lambda,\lambda,\lambda,\lambda)$.}

When the Schubert problem has the form $\blambda=\lambda^4$, there are three osculation
types, $r_\lambda=4$, $r_\lambda=2$, and $r_\lambda=0$.
Every Schubert problem of this type we studied has interesting structure in its numbers of
real solutions.
Table~\ref{Ta:21e4=8} 
%%%%%%%%%%%%%%%%%%%%%%%%%%%%%%%%%%%%%%%%%%%%%%%%%%%%%%%%%%%%%%%%%%%%%%%%%%%%%%%%%
% From old experiment
% 21^4=8
%
\begin{table}[htb]
  \caption{Frequency table for $\TI^4=8$ in $\Gr(3,7)$.}
  \label{Ta:21e4=8}
 \begin{tabular}{|c||r|r|r|r|r||r|}\hline
   \multirow{2}{*}{$r_{\TIs}$}&
   \multicolumn{5}{c||}{\textbf{\# Real Solutions}}&
   \multirow{2}{*}{\textbf{Total}}\\\cline{2-6}
   &$\boldsymbol{0}$&$\boldsymbol{2}$&$\boldsymbol{4}$&
     $\boldsymbol{6}$&$\boldsymbol{8}$&\\\hline\hline
   \textbf{4}&&&&&10000&\textbf{10000}\\\hline
   \textbf{2}&3590&292&6118&&&\textbf{10000}\\\hline
   \textbf{0}&&&10000&&&\textbf{10000}\\\hline 
 \end{tabular}
\end{table}
%%%%%%%%%%%%%%%%%%%%%%%%%%%%%%%%%%%%%%%%%%%%%%%%%%%%%%%%%%%%%%%%%%%%%%%%%%%%%%%%%
shows the results for the Schubert problem $\TIs^4=8$ in $\Gr(3,7)$.
The structure of this table is similar to Table~\ref{Ta:31e4=9} for the Schubert problem
$\ThI^4=9$ in $\Gr(4,8)$. 
%%%%%%%%%%%%%%%%%%%%%%%%%%%%%%%%%%%%%%%%%%%%%%%%%%%%%%%%%%%%%%%%%%%%%%%%%%%%%%%%%
% From old experiment
% 31^4=9
%
\begin{table}[htb]
 \caption{Frequency table for $\ThI^4=9$ in $\Gr(4,8)$.}
  \label{Ta:31e4=9}
 \begin{tabular}{|c||r|r|r|r|r||r|}\hline
   \multirow{2}{*}{$r_{\ThIs}$}&
   \multicolumn{5}{c||}{\textbf{\# Real Solutions}}&
   \multirow{2}{*}{\textbf{Total}}\\\cline{2-6}
   &$\boldsymbol{1}$&$\boldsymbol{3}$&$\boldsymbol{5}$&
     $\boldsymbol{7}$&$\boldsymbol{9}$&\\\hline\hline
   \textbf{4}&&&&&7500&\textbf{7500}\\\hline
   \textbf{2}&4995&13&4692&&&\textbf{7500}\\\hline
   \textbf{0}&&&7500&&&\textbf{7500}\\\hline 
 \end{tabular}
\end{table}
%%%%%%%%%%%%%%%%%%%%%%%%%%%%%%%%%%%%%%%%%%%%%%%%%%%%%%%%%%%%%%%%%%%%%%%%%%%%%%%%%
Both of these were computed in~\cite{Orig_Lower} which inspired
the more comprehensive experiment we have been discussing.

Understanding these tables led Purbhoo~\cite{Pu13} to study the number fixed points in a
fiber of the Wronski map under the action of a cyclic or dihedral group.
His Theorem 3.15 gives a formula for the number of real solutions to instances of
Schubert problems $(\lambda,\lambda,\lambda,\lambda)$ with osculation type $r_\lambda=0$.
The number of complex solutions to this problems is a particular set of Young tableaux,
and Purbhoo's formula is the number of these Young tableaux that are fixed under an
involution based on tableaux switching~\cite{BSS}. 
Example 3.16 of~\cite{Pu13} gives the computation that this number is 4 for $\TIs^4=8$ in
$\Gr(3,7)$, as we saw in Table~\ref{Ta:21e4=8}.
Similarly, it is an exercise that this number is 5 for $\ThIs^4=9$ in $\Gr(4,8)$.

Purbhoo's result may be applied to Schubert problems in the family of Schubert
problems $(a,0)^4=a{+}1$ in $\Gr(2,2a)$ when $r_{(a,0)}=0$.
As we see for the Schubert problem $\I^4=2$ of Table~\ref{Ta:two_tables} and 
$\Th^4=4$ in  $\Gr(2,8)$ of Table~\ref{Ta:3e4=4}, 
%%%%%%%%%%%%%%%%%%%%%%%%%%%%%%%%%%%%%%%%%%%%%%%%%%%%%%%%%%%%%%%%%%%%%%%%%%%%%%%%%
% result_id = 165
% 73.817 GHz-days
%
\begin{table}[htb]
 \caption{Frequency tables for $\Th^4 = 4$ in $\Gr(2,8)$.}
  \label{Ta:3e4=4}

 \begin{tabular}{|c||r|r|r||r|}\hline
   \multirow{2}{*}{$r_{\Ths}$}&
   \multicolumn{3}{c||}{\textbf{\# Real Solutions}}&
   \multirow{2}{*}{\textbf{Total}}\\\cline{2-4}
   &$\boldsymbol{0}$&$\boldsymbol{2}$&$\boldsymbol{4}$&\\\hline\hline
   \textbf{4}&&&200000&\textbf{200000}\\\hline
   \textbf{2}&32765&103284&63951&\textbf{200000}\\\hline
   \textbf{0}&&&200000&\textbf{200000}\\\hline 
 \end{tabular}
\end{table}
%%%%%%%%%%%%%%%%%%%%%%%%%%%%%%%%%%%%%%%%%%%%%%%%%%%%%%%%%%%%%%%%%%%%%%%%%%%%%%%%%
when there are no real points of osculation, these problems appear to have $a{+}1$ real
solutions.
That is in fact always the case, as we now explain.

The solutions to $(a,0)^4=a{+}1$ are enumerated by Young tableaux
of shape $(2a,2a)$ filled with $a$ copies of each of the numbers $1$, $2$, $3$, and $4$.
Since the numbers $1$ must fill the first $a$ positions in the first row and the numbers
$4$ must fill the last $a$ positions in the second row, the only choice is how many
numbers 2 are in the first row.
There are $a{+}1$ choices, so there are $a{+}1$ such tableaux.
Here are the four tableaux for the problem $\Th^4=4$.
\[
  \TheFo{2}{2}{2}{3}{3}{3}\qquad 
  \TheFo{2}{2}{3}{2}{3}{3}\qquad
  \TheFo{2}{3}{3}{2}{2}{3}\qquad
  \TheFo{3}{3}{3}{2}{2}{2}\,.
\]
Purbhoo's switching involution switches the subtableaux consisting of the 1s with that of 
the 2s, and that of the 3s with that of the 4s.
However, the properties of switching (see~\cite{BSS} or~\cite{Pu13}) imply that every such
tableaux is fixed under this involution, which implies that all solutions will be real for
$(a,0)^4=a{+}1$ with osculation type $r_{(a,0)}=0$.

Despite this understanding for $r_\lambda=4$ and $r_\lambda=0$, we do not understand the
possible numbers of real solutions when $r_\lambda=2$ for Schubert problems
$(\lambda,\lambda,\lambda,\lambda)$.

% ! &result_id=490  (Maybe was in mod four paper?

%&result_id=574
%%%%%%%%%%%%%%%%%%%%%%%%%%%%%%%%%%%%%%%%%%%%%%%%%%%%%%%%%%%%%%%%%%%%%%%%%%
\providecommand{\bysame}{\leavevmode\hbox to3em{\hrulefill}\thinspace}
\providecommand{\MR}{\relax\ifhmode\unskip\space\fi MR }
\providecommand{\MRhref}[2]{%
  \href{http://www.ams.org/mathscinet-getitem?mr=#1}{#2}
}
\providecommand{\href}[2]{#2}

%%%%%%%%%%%%%%%%%%%%%%%%%%%%%%%%%%%%%%%%%%%%%%%%%%%%%%%%%%%%%%%%%%%%%%%%%%

\end{document}